\documentclass[10pt, xcolor = {usenames,dvipsnames}]{article}

\usepackage[english]{babel}
\usepackage[utf8]{inputenc} 
\usepackage{amsfonts}
\usepackage[square,sort,comma,numbers]{natbib}
\bibpunct{[}{]}{;}{a}{,}{,}
\usepackage[hidelinks]{hyperref}

\usepackage{tikz}
\usetikzlibrary{calc}

\usepackage{mathtools}
\usepackage{amsthm}
\usepackage{amssymb}
\usepackage{amsmath}
\usepackage{needspace}
\usepackage{etoolbox}
\usepackage{lipsum}
\usepackage{authblk}
\usepackage{mathrsfs}
\usepackage{bbold}
\usepackage{graphicx}
\usepackage{subcaption}
\usepackage{color,soul}

\newcounter{theo}
\newtheorem{lmm}[theo]{Lemma}
\newtheorem{thr}{Theorem}
\newtheorem{stm}[theo]{Statement}
\newtheorem{crl}[theo]{Corollary}
\newtheorem{prp}[theo]{Proposition}

\newtheorem{defn}[theo]{Definition}
\newtheorem{rmrk}[theo]{Remark}

\newtheorem{thr_old}[theo]{Theorem}

\DeclareMathOperator{\arccosh}{arccosh}

\newcommand{\R}{\mathbb{R}}
\newcommand{\Compl}{\mathbb{C}}

\newcommand{\T}{\mathbb{T}}
\newcommand{\D}{\mathbb{D}}
\newcommand{\Prob}{\mathbb{P}}

\newcommand{\eps}{\varepsilon}

\newcommand\CnjCl[1]{#1}

\newcommand{\E}{\mathbb{E}}

\newcommand{\Diff}{\mathrm{Diff}}
\newcommand{\DiffG}{\widetilde{\mathrm{Diff}}}
\newcommand{\SL}{\mathrm{PSL}}

\newcommand{\measPin}[3]{\mathscr{L}_{#1}^{#2, #3}}
\newcommand{\measN}[2]{\measPin{#1}{#2}{0}}
\newcommand{\FMeas}[1]{\mathscr{M}_{#1}}
\newcommand{\FMeasSL}[1]{\widetilde{\mathscr{M}}_{#1}}

\newcommand{\TechMeas}{\mathcal{P}}

\newcommand{\Schw}{\mathcal{S}}

\newcommand{\A}{\mathsf{P}}
\newcommand{\B}{\mathsf{Q}}
\newcommand{\sh}[2]{\mathsf{R}^{#2} #1}
\renewcommand*{\d}{\mathop{}\!\mathrm{d}}
\newcommand{\WS}[3]{\mathcal{B}_{#1}^{\, #2, #3}} %
\newcommand{\Cfree}{C_{0,\text{free}}}
\newcommand{\Dom}{\Omega}
\newcommand{\Real}{\Re}

\newcommand{\He}{\mathrm{H}}

\makeatletter
\def\obs#1#2#3{\def\temp@expr{\big(#1; #2, #3\big)}\mathcal{O}\obs@}
\def\obs@{%
	\@ifnextchar{_}{\obs@sub}{
	\@ifnextchar{^}{\obs@sup}{\temp@expr}}}
\def\obs@sub#1#2{_{#2}\obs@}
\def\obs@sup#1#2{^{\, #2}\obs@}
\makeatother

\newcommand{\fcolor}{blue}

\title{Probabilistic Correlation Functions\\
 of the Schwarzian Field Theory}
\author{Ilya Losev\footnote{Mathematical Institute, University of Oxford, Andrew Wiles Building, Radcliffe Observatory Quarter, Woodstock Road, Oxford, OX2 6GG, UK.
E-mail: \url{ilya.losev@maths.ox.ac.uk}.}}

\date{May 25, 2026}

\begin{document}

\maketitle

\begingroup
\renewcommand\thefootnote{}
\addtocounter{footnote}{-1}
\endgroup

\abstract{
We study correlation functions of the probabilistic Schwarzian Field Theory.
We compute cross-ratio correlation functions exactly in the case when the corresponding Wilson lines do not intersect, confirming predictions made in the physics literature via limit of the conformal bootstrap and the DOZZ formula. 
Moreover, we prove that these correlation functions characterise the measure uniquely.
We use them to define and compute the stress-energy tensor correlation functions,
and demonstrate, in particular, that these agree with the results obtained earlier by formal differentiation of the partition function.
}

\section{Introduction and main results}
\label{sect_intro_main_results}
\subsection{Introduction}
Recently, the Schwarzian Field Theory has attracted a lot of attention in the physics literature. 
It appears in the study of the AdS/CFT correspondence as a holographic dual of Jackiw–Teitelboim (JT) gravity in the disk \citep{SaadShenkerStanford2019, NearlyAdS, JT_Wilson_Line}.
The Schwarzian Field Theory also emerges in the low energy limit of the Sachdev–Ye–Kitaev (SYK) random matrix model (e.g. see \citep{MaldacenaStanford} and \citep{KitaevJosephine}), and is related to Liouville CFT, infinite dimensional symplectic geometry, representation theory of the Virasoro algebra and 2D Yang-Mills.

In our companion paper \citep*{BLW} we have defined and proved uniqueness and existence of a finite measure on $\Diff^1(\T)/\SL(2, \R)$, which corresponds to the Schwarzian Field Theory. 
The existence part follows the plan proposed in physics in \citep{BelokurovShavgulidzeExactSolutionSchwarz, BelokurovShavgulidzeCorrelationFunctionsSchwarz}.
We have also rigorously computed its partition function (i.e. total mass) using methods of stochastic analysis.
The obtained result for the partition function agrees with the formula derived non-rigorously in \citep{StanfordWittenFermionicLocalization} using a formal application of the Duistermaat–Heckman theorem on the infinite dimensional symplectic space $\Diff^1(\T)/\SL(2, \R)$.

In this paper we continue the study of this measure 
and rigorously compute a natural class of correlation functions of cross-ratio observables.
These correlation functions were originally formally derived in physics in \citep{ConformalBootstrap} by taking a $c\to \infty$ limit of two dimensional Liouville CFT and, in particular, a limit of the DOZZ formula.
Our results match the formulae obtained in \citep{ConformalBootstrap}. 
Moreover, we prove that the computed correlation functions determine the measure uniquely, which further confirms that the measure defined in \citep*{BLW} corresponds to the Schwarzian Field Theory studied in the physics literature.
This uniqueness theorem complements characterisation of the measure via change of variables formula which was obtained in \citep*{BLW}.
In addition, we show that using these observables we can make sense of and compute the stress-energy tensor correlation functions, using a procedure similar to one proposed in \citep{ConformalBootstrap}. 
We also prove that the stress-energy tensor correlation functions calculated this way agree with values obtained by differentiation of the partition function, as proposed in \citep{StanfordWittenFermionicLocalization} and carried out in \citep*{BLW}. 
This, in particular, proves a conjecture from \citep{ConformalBootstrap}.

\medskip

Formally, the measure corresponding to the Schwarzian Field Theory is supported on the topological space $\Diff^1(\T)/\SL(2, \R)$ with density (see \citep[(1.1)]{StanfordWittenFermionicLocalization})
\begin{equation}\label{eq:1}
\d\FMeas{\sigma^2}\big(\phi\big) = 
\exp\left\{+\frac{1}{\sigma^2}\int_{\T} 
\left[\Schw_\phi(\tau)+2\pi^2\phi'^{\, 2}(\tau)\right] \d\tau \right\}
\frac{\prod_{\tau \in \T}\frac{\d\phi(\tau)}{\phi'(\tau)}}{\SL(2, \R)} ,
\end{equation} 
where $\Schw_{\phi}(\tau)$ is the Schwarzian derivative of $\phi$ given by
\begin{equation}
\Schw_{\phi}(\tau) = \Schw(\phi, \tau) = \left(\frac{\phi''(\tau)}{\phi'(\tau)}\right)'- \frac{1}{2} \left(\frac{\phi''(\tau)}{\phi'(\tau)}\right)^2.
\end{equation}
Here, $\T = [0, 1]/\{0\sim 1\}$ is the unit circle, $\Diff^1(\T)$ is the space of $C^1$ orientation preserving diffeomorphisms of $\T$, and $\SL(2, \R)$ is the group of M\"{o}bius transformations of the unit disk (i.e. conformal isomorphisms of the unit disk) restricted to the boundary which is identified with $\T$.
The group $\SL(2, \R)$ acts on $\Diff^1(\T)$ by post-compositions. 
Following \citep{StanfordWittenFermionicLocalization} we call it a right action, since in \citep{StanfordWittenFermionicLocalization} it is interpreted as an action on the inverse elements.
We denote the quotient of $\Diff^1(\T)$ by this action of $\SL(2, \R)$ by $\Diff^1(\T)/\SL(2, \R)$, with the topology inherited from $C^1$ topology on $\Diff^1(\T)$. Heuristically, the formal density \eqref{eq:1} only depends on the orbit of this action and the quotient by $\SL(2,\R)$ therefore makes sense.
In Section \ref{sectMeasConstructAndProp} we recall the rigorous construction of the measure corresponding to the Schwarzian Field Theory, carried out in \citep*{BLW}.

In our investigation of correlation functions we consider bi-local observables given by cross-ratios 
\begin{equation}\label{eqDefObs}
\obs{\phi}{s}{t} = \frac{\pi\sqrt{\phi'(t)\phi'(s)}}{\sin\big(\pi\big[\phi(t)-\phi(s)\big]\big)}, \qquad s,t\in \T.
\end{equation}
These observables are defined for $\phi\in \Diff^1(\T)$ and are invariant under the $\SL(2, \R)$ action, see Proposition \ref{prpObsSLInvar}. 
Therefore, these cross-ratios naturally induce well-defined observables on $\Diff^1(\T)/\SL(2, \R)$, for which we use the same notation.
Correlations of these observables are predicted to correspond to boundary-anchored Wilson line correlators in the JT gravity \citep{Schwarzian_Wilson_Line}.
It is also predicted that these correlations describe the boundary-to-boundary propagator of a massive particle in the metric formulation of JT gravity \citep{JT_Wilson_Line}.
These observables are also known to be related to the Green's function of the SYK model \citep{Bagrets_SYK_as_LQG, SSS_ramp_SYK}.

We rigorously compute cross-ratio correlation functions in the case when Wilson lines (lines connecting $s$ and $t$ on the circle) do not intersect, see Figure \ref{fig:observables_diagram_def_fourier}. 
In this case we call a set of observables \textit{non-interlaced}, see Section \ref{sect_Diagram_circ}.
Moreover, we prove that the computed correlation functions characterise the measure uniquely.
Our methods rely on stochastic analysis and, as far as we know, differ from the methods existing in the physics literature.
In particular, we prove our results directly, without relying on the two dimensional Liouville Theory, and the DOZZ formula, which was rigorously proved  only recently in \citep{DOZZ}.

Furthermore, these observables are related to the stress-energy tensor of the Schwarzian Field Theory. 
Formally, the stress-energy tensor is given by the Schwarzian derivative $\Schw\big(\tan (\pi \phi), \cdot\big) = \Schw_\phi(\tau)+2\pi^2\phi'^{\, 2}(\tau)$, which is connected to observables \eqref{eqDefObs}, at least for sufficiently smooth $\phi$, by
\begin{equation}
\left(\frac{\pi\sqrt{\phi'(t)\phi'(s)}}{\sin\big(\pi\big[\phi(t)-\phi(s)\big]\big)} \right)^2
= (t-s)^{-2}+
\frac{1}{6}\, \Schw\big(\tan (\pi \phi), \tau \big)
+o(1), \qquad 
\text{as } t-s \to 0.
\end{equation}
Even though regularity of diffeomorphisms under $\d\FMeas{\sigma^2}$ is not sufficient for the Schwarzian derivative to exist, this relation can be used on the level of correlation functions. 
In Section \ref{sect_StressEnergy} we show that the corresponding limits of correlation functions exist, and we compute them explicitly.
We show that this limit coincides with formulae obtained by differentiating the partition function in \citep*{BLW}, in particular confirming the conjecture from \citep{ConformalBootstrap}, which was checked numerically there.

\subsection{Diagrammatic representation on $\T$}
\label{sect_Diagram_circ}
In order to formulate the main results we develop a diagrammatic language.
To a given set of observables we associate a diagram in a disk, where every observable is represented with a chord.

\medskip

Let $\big\{\obs{\phi}{s_{j}}{t_j}\big\}_{j=1}^{N}$ be a set of observables. 
Given this set we draw a unit circle and all the points $\{s_j\}_{j=1}^N$ and $\{t_j\}_{j=1}^N$ on it. 
For all $1\leq j \leq N$ we connect the point $s_j$ with the point $t_j$ with a line segment.
These line segments also correspond to Wilson lines, see \citep{Schwarzian_Wilson_Line, JT_Wilson_Line}.

\begin{defn}
We say that a set of observables $\big\{\obs{\phi}{s_{j}}{t_j}\big\}_{j=1}^{N}$ is non-interlaced if interiors of all drawn chords (Wilson lines) in the corresponding diagram are pairwise non-intersecting.
\end{defn}
Figure \ref{fig:interlacement} gives examples of graphic representations of two observables being interlaced and non-interlaced.

\begin{figure}
\centering
\begin{subfigure}[b]{0.3\textwidth}
\centering
\begin{tikzpicture}
  \draw (0,0) circle (1.3);
	\node[shape=circle,fill=black, scale=0.5,label={90:$s_p$}] (ns1) at (90:1.3) {};
	\node[shape=circle,fill=black, scale=0.5,label={180:$t_p$}] (nt1) at (180:1.3) {};
	\node[shape=circle,fill=black, scale=0.5,label={270:$s_q$}] (ns2) at (270:1.3) {};
	\node[shape=circle,fill=black, scale=0.5,label={360:$t_q$}] (nt2) at (360:1.3) {};
  \draw (ns1) -- (nt1);
  \draw (ns2) -- (nt2);
\end{tikzpicture}
\caption{Non-interlaced pair}
\end{subfigure}
\begin{subfigure}[b]{0.3\textwidth}
\centering
\begin{tikzpicture}
  \draw (0,0) circle (1.3);
	\node[shape=circle,fill=black, scale=0.5,label={90:$s_p=s_q$}] (ns1) at (90:1.3) {};
	\node[shape=circle,fill=black, scale=0.5,label={210:$t_p$}] (nt1) at (210:1.3) {};
	\node[shape=circle,fill=black, scale=0.5,label={330:$t_q$}] (nt2) at (330:1.3) {};
	\node[label={270:$\,_{\,}$}] (ns5) at (270:1.45) {};
  \draw (ns1) -- (nt1);
  \draw (ns1) -- (nt2);
\end{tikzpicture}
\caption{Non-interlaced pair}
\end{subfigure}
\begin{subfigure}[b]{0.3\textwidth}
\centering
\begin{tikzpicture}
  \draw (0,0) circle (1.3);
	\node[shape=circle,fill=black, scale=0.5,label={90:$s_p$}] (ns1) at (90:1.3) {};
	\node[shape=circle,fill=black, scale=0.5,label={270:$t_p$}] (nt1) at (270:1.3) {};
	\node[shape=circle,fill=black, scale=0.5,label={180:$s_q$}] (ns2) at (180:1.3) {};
	\node[shape=circle,fill=black, scale=0.5,label={0:$t_q$}] (nt2) at (0:1.3) {};
  \draw (ns1) -- (nt1);
  \draw (ns2) -- (nt2);
\end{tikzpicture}
\caption{Interlaced pair}
\end{subfigure}

\caption{An example of two observables.}\label{fig:interlacement}
\end{figure}
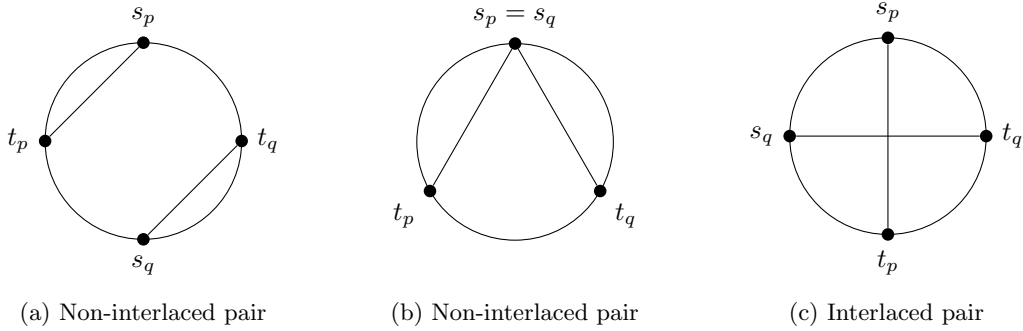

Let $\big\{\obs{\phi}{s_{j}}{t_j}\big\}_{j=1}^{N}$ be a set of non-interlaced observables. 
The $N$ drawn chords on the corresponding diagram divide the unit disk into $N+1$ connected domains.
We number them with integers from $1$ to $N+1$.
For each $m \in \{1, \ldots N+1\}$ we associate a Fourier variable $k_m$ to domain number $m$. We also let $\tau_m$ be the total length of all arcs (parts of the initial circle) which form the boundary of the $m$-th domain (see Figure \ref{fig:observables_diagram_def_fourier}).

For each $j \in \{1, \ldots N\}$ we define $w_1(j)$ and $w_2(j)$ to be the Fourier variables corresponding to the domains that contain the line segment connecting $s_j$ with $t_j$ in their boundaries (see Figure \ref{fig:observables_diagram_obs_fourier}). All formulae will be symmetric in $w_1$ and $w_2$, so the exact order does not matter.

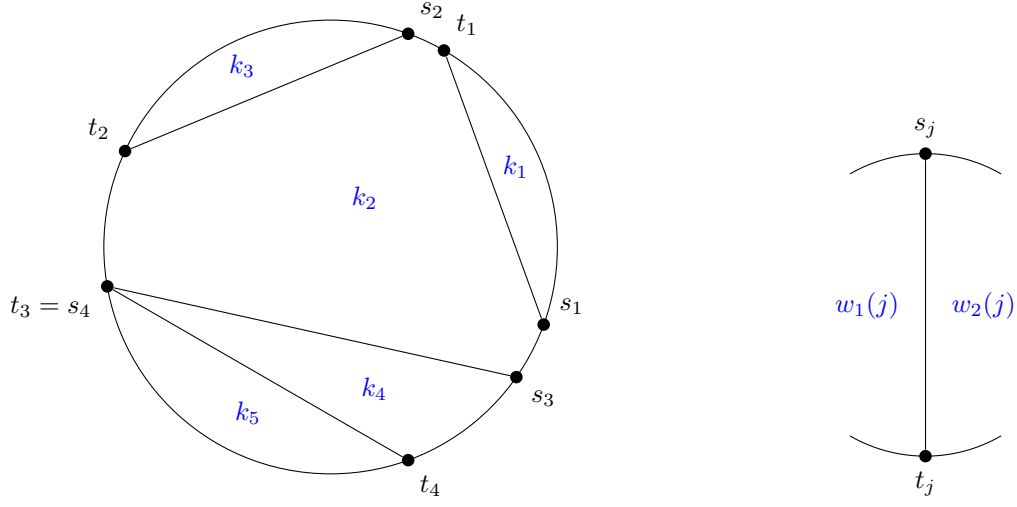
\begin{figure}[tbh]\centering
\begin{subfigure}[t]{0.55\textwidth}
\centering
\begin{tikzpicture}
\newcommand\radius{3}
  \draw (0,0) circle (\radius);
	\node[shape=circle,fill=black, scale=0.5,label={10:$s_1$}] (ns1) at (-20:\radius) {};
	\node[shape=circle,fill=black, scale=0.5,label={60:$t_1$}] (nt1) at (60:\radius) {};
	\node[shape=circle,fill=black, scale=0.5,label={70:$s_2$}] (ns2) at (70:\radius) {};
	\node[shape=circle,fill=black, scale=0.5,label={155:$t_2$}] (nt2) at (155:\radius) {};
	\node[shape=circle,fill=black, scale=0.5,label={-35:$s_3$}] (ns3) at (-35:\radius) {};
	\node[shape=circle,fill=black, scale=0.5,label={190:$t_3=s_4$}] (nt3) at (190:\radius) {};
	\node[shape=circle,fill=black, scale=0.5,label={-70:$t_4$}] (nt4) at (-70:\radius) {};
  \draw (ns1) -- (nt1);
  \draw (ns2) -- (nt2);
  \draw (ns3) -- (nt3);
  \draw (nt3) -- (nt4);
  \node[label={20:$\color{\fcolor} k_1$}] (nf1) at (20:0.72*\radius) {};
  \node[label={45+22.5:$\color{\fcolor} k_2$}] (nf2) at (45+22.5:0.09*\radius) {};
  \node[label={90+22.5:$\color{\fcolor} k_3$}] (nf3) at (90+22.5:0.72*\radius) {};
  \node[label={280:$\color{\fcolor} k_4$}] (nf5) at (280:0.5*\radius) {};
  \node[label={225+22.5:$\color{\fcolor} k_5$}] (nf6) at (225+22.5:0.65*\radius) {};
\end{tikzpicture}
\caption{An example of a diagram. Here,\\ 
$\tau_1 = (t_1-s_1)$, \\
$\tau_2 = (s_1-s_3)+(s_2-t_1)+(t_3-t_2)$, \\
$\tau_3 = (t_2-s_2)$, \\
$\tau_4 = (s_3-t_4)$,\\ 
$\tau_5 = (t_4-s_4)$.
}\label{fig:observables_diagram_def_fourier}
\end{subfigure}
\hspace{0.05\textwidth}
\begin{subfigure}[t]{0.36\textwidth}
\centering
\begin{tikzpicture}
\newcommand\radius{2}
  \draw [domain=60:120] plot ({\radius*cos(\x)},{\radius*sin(\x)});
  \draw [domain=240:300] plot ({\radius*cos(\x)},{\radius*sin(\x)});
  \node[shape=circle,fill=black, scale=0.5,label={90:$s_j$}] (ns1) at (90:\radius) {};

	\node[shape=circle,fill=black, scale=0.5,label={270:$t_j$}] (nt1) at (270:\radius) {};
	\node[label={0:$\color{\fcolor} w_2(j)$}] (nf1) at (0:0.05*\radius) {};
	\node[label={180:$\color{\fcolor} w_1(j)$}] (nf2) at (180:0.05*\radius) {};
	
  \draw (ns1) -- (nt1);

\end{tikzpicture}
\caption{We use $w_1(j)$ and $w_2(j)$ to denote the Fourier variables $k_m$ corresponding to the domains that lie on both sides of the line segment that connects $s_j$ with $t_j$. \\
In Figure \ref{fig:observables_diagram_def_fourier}, for example, $w_1(3) = k_2$ and $w_2(3) = k_4$ (or vice versa).}
\label{fig:observables_diagram_obs_fourier}
\end{subfigure}
\caption{A diagrammatic representation of Fourier variables.}
\end{figure}

\begin{defn}
For all $l, k, w \in \R$ we define $\Gamma(l\pm ik\pm iw)$ as
\begin{equation}
\Gamma(l\pm ik\pm iw) := 
\Gamma(l + ik + iw)  \Gamma(l + ik - iw)  \Gamma(l - ik + iw)  \Gamma(l - ik - iw) .
\end{equation}
\end{defn}
\begin{rmrk}
For any $l>0$, the function $\Gamma(l\pm ik\pm iw) $ is non-negative, since  $\Gamma(z)\Gamma(\bar{z}) = |\Gamma(z)|^2$.
\end{rmrk}

\begin{rmrk}\label{rmrkGammaFormula}
It is well-known (see e.g. \citep[8.332, 8.331.1]{GradshteynRyzhik}) that for all positive integers $n$ and $x\in \R$,
\begin{align}
\left|\Gamma(1+n+ix)\right|^2 &= \frac{\pi x}{\sinh (\pi x)}\prod_{k=1}^n\left(k^2+x^2\right),\\
\left|\Gamma\left(\tfrac{1}{2}+n+ix\right)\right|^2 &= \frac{\pi }{\cosh (\pi x)}\prod_{k=1}^n\left(\left(k-\tfrac{1}{2}\right)^2+x^2\right).
\end{align}
\end{rmrk}

Using the diagrammatic language we can now formulate the main result about the correlation functions, confirming predictions from \citep{ConformalBootstrap}.
\begin{thr} \label{thrMainCorrelations}
For $N\geq 0$ let $\big\{\obs{\phi}{s_{j}}{t_j}\big\}_{j=1}^{N}$ be a set of non-interlaced observables, and $\{l_j\}_{j=1}^N$ be a set of positive integers. 
Let also $\{k_m\}_{m=1}^{N+1}$ and $\{\tau_m\}_{m=1}^{N+1}$ be as above.
Then 
\begin{multline}
\int\limits_{{\Diff^1(\T)/\SL(2, \R)}} \prod_{j=1}^N \obs{\phi}{s_j}{t_j}^{l_j}\d\FMeas{\sigma^2}(\CnjCl{\phi})
=
\int_{\R_+^{N+1}} 
\prod_{j=1}^N \frac{\Gamma\left(\dfrac{l_j}{2}\pm i w_1(j) \pm i w_2(j)\right)}{2\pi^2 \, \Gamma(l_j)}
\cdot\left(\frac{\sigma^2}{2}\right)^{l_j}\\
\times
\prod_{m=1}^{N+1}\exp\left(-\frac{\tau_m \sigma^2}{2}\cdot k_m^2\right)  \sinh(2\pi k_m) \, 2 k_m \d k_m,
\end{multline}
where the integral on the right-hand side converges absolutely.
\end{thr}

\begin{rmrk}
Taking $N=0$ gives a formula for the total mass of the measure $\d\FMeas{\sigma^2}$, which agrees with the result first proposed in \citep{StanfordWittenFermionicLocalization} and obtained rigorously in \citep*{BLW}.
\end{rmrk}

\begin{rmrk}
This is also true for $l_j=0$ if we interpret 
\begin{equation}
\frac{\Gamma\big(\pm i k \pm i w\big)}{2\pi^2\, \Gamma(0)}:=\frac{\delta(k-w)}{\sinh(2\pi k)\, 2 k },
\end{equation}
where $\delta(\cdot)$ is the delta function.
\end{rmrk}

As an immediate corollary we obtain formulae for moments of the observables.
\begin{crl}\label{crlMainObsMoments}
For positive integers $l$ the moments of observables are given by
\begin{multline}
\int \obs{\phi}{s}{t}^{l}\d\FMeas{\sigma^2}(\CnjCl{\phi}) 
= 
\int_{\R_+^{2}} 
\frac{\Gamma\big(\frac{l}{2} \pm i k_1 \pm i k_2\big)}{2\pi^2\, \Gamma(l)}\cdot\left(\frac{\sigma^2}{2}\right)^l
\\
\times
\exp\left(-\frac{(t-s)\sigma^2}{2}\cdot k_1^2 -\frac{\big(1-(t-s)\big)\sigma^2}{2}\cdot k_2^2\right)  
\sinh(2\pi k_1)\, 2 k_1 \sinh(2\pi k_2)\, 2 k_2 \d k_1 \d k_2.
\end{multline}
\end{crl}

As another useful corollary we show that all observables have exponential moments.
\begin{prp}\label{prpExpMoment}
For any $\sigma>0$ and any $s\neq t \in \T$, 
\begin{equation}
\int  \exp\left\{ \frac{8}{\sigma^2}\, \obs{\phi}{s}{t} \right\} \d\FMeas{\sigma^2}(\CnjCl{\phi}) < \infty.
\end{equation}
\end{prp}
\begin{rmrk}\label{rmrkExpOpt}
The constant $8/\sigma^2$ in the exponential is sharp.
\end{rmrk}

We demonstrate that the computed correlation functions determine the measure uniquely.
Recall that the topology of $\Diff^1(\T)/\SL(2,\R)$ is inherited from $C^1$ topology on $\Diff^1(\T)$.
\begin{thr}\label{thrMainUniq}
Let $\TechMeas$ be a Borel measure on $\Diff^1(\T)/\SL(2,\R)$. 
Suppose that 
for any $N\geq 0$, any set $\big\{\obs{\phi}{s_{j}}{t_j}\big\}_{j=1}^{N}$ of non-interlaced observables, and any set $\{l_j\}_{j=1}^N$ of positive integers we have 
\begin{equation}
\int \prod_{j=1}^N \obs{\phi}{s_j}{t_j}^{l_j} \d\TechMeas(\CnjCl{\phi})
=
\int \prod_{j=1}^N \obs{\phi}{s_j}{t_j}^{l_j} \d\FMeas{\sigma^2}(\CnjCl{\phi}).
\end{equation}
Then for any Borel set $A\subset \Diff^1(\T)/\SL(2, \R)$ we have
\begin{equation}
\TechMeas(A) = \FMeas{\sigma^2}(A).
\end{equation}
\end{thr}

\begin{rmrk}
In this work we only consider non-interlaced observables. 
It would be interesting to compute general correlation functions. 
It is expected that we get similar formulae, where every crossing contributes an extra factor, called $6j$-symbol, see \citep{ConformalBootstrap, Schwarzian_Wilson_Line, JT_Wilson_Line}.
Theorem \ref{thrMainUniq} tells us that, in principle, it should be possible to deduce this from the correlation functions already computed in Theorem \ref{thrMainCorrelations}.
\end{rmrk}
The theorem above further confirms that the measure defined and constructed in \citep*{BLW} coincides with the Schwarzian Field Theory studied in the physics literature.
This uniqueness theorem complements characterisation of the Schwarzian measure via change of variables formula established in \citep*{BLW}.

\subsection{Stress-energy tensor}\label{sect_StressEnergy}
In this section we make sense of the stress-energy tensor correlation functions studied in \citep{StanfordWittenFermionicLocalization}.
Formally, in Schwarzian Field Theory the stress-energy tensor is given by the Schwarzian derivative $\Schw\big(\tan (\pi \phi), \cdot\big)$. 
However, the Schwarzian derivative is a non-linear functional of the field which is apriori not well-defined on the support of $\d \FMeas{\sigma^2}$ due to lack of sufficient regularity (essentially, the measure is supported on diffeomorphisms of regularity $C^{3/2-}$).
Nevertheless, we show that it is possible to make sense of its correlation functions.

Our approach is based on the fact that, at least for sufficiently smooth $\phi$, we can express the stress-energy tensor $\Schw\big(\tan (\pi \phi), \cdot\big)$ using cross-ratio observables $\obs{\phi}{s}{t}$ by
\begin{equation}\label{eq_Schw_using_Obs}
\Schw\big(\tan (\pi \phi), \tau \big)
= 6\lim_{\substack{s\to \tau-\\
t\to \tau+}} 
\left(\obs{\phi}{s}{t}^2-(t-s)^{-2}\right).
\end{equation}
Even though the limit on the right-hand side might not exist on the support of $\d \FMeas{\sigma^2}$, we can \textit{define} correlation functions of the stress-energy tensor $\Schw\big(\tan (\pi \phi), \cdot \big)$ by taking the corresponding limit of the correlation functions of the expression appearing on the right-hand side of \eqref{eq_Schw_using_Obs}. 
However, in order to match the resulting correlation functions to formulae obtained by formal differentiation of the partition function in \citep{StanfordWittenFermionicLocalization} and \citep*{BLW} we need to add an extra shift by a constant. 
This is often expected when dealing with correlations of classically ill-posed composite fields. 

In the following Theorem we compute the joint correlation functions of cross-ratio observables and stress-energy tensors at non-coinciding points.  
\begin{thr}\label{thrStressEnergy}
Let $N\geq 0$, $M\geq 0$ be integers, $\big\{\obs{\phi}{s_{j}}{t_j}\big\}_{j=1}^{N}$ be a set of non-interlaced observables, $\left\{l_j\right\}_{j=1}^N$ be non-negative integers,
$\left\{r_p\right\}_{p=1}^{M}$  be distinct points on the unit circle $\T$, which are different from $\left\{s_j\right\}_{j=1}^N$ and $\left\{t_j\right\}_{j=1}^N$. 
For $t\in \T$ we write $k(t)$ for the Fourier variable which corresponds to the diagram domain whose boundary contains $t$.  
Then for any $\sigma>0$,
\begin{multline}
\lim_{\eps_1, \ldots \eps_M\to 0+}\int 6^M \prod_{p=1}^M \left(\obs{\phi}{r_p}{r_p+\eps_p}^2-\eps_p^{-2}-\frac{\sigma^4}{240}\right)
\prod_{j=1}^{N} \obs{\phi}{s_j}{t_{j}}^{l_j}
\d\FMeas{\sigma^2}(\CnjCl{\phi})\\ 
=
\int_{\R_+^{N+1}}\sigma^{4M} \prod_{p=1}^M k^2(r_p)
\prod_{j=1}^N \frac{\Gamma\left(\dfrac{l_j}{2}\pm i w_1(j) \pm i w_2(j)\right)}{2\pi^2 \, \Gamma(l_j)}
\cdot\left(\frac{\sigma^2}{2}\right)^{l_j}\\
\times
\prod_{m=1}^{N+1}\exp\left(-\frac{\tau_m \sigma^2}{2}\cdot k_m^2\right)  \sinh(2\pi k_m) \, 2 k_m \d k_m.
\end{multline}
\end{thr}

\begin{rmrk}
The Schwarzian derivative also appears in the expression of the stress-energy tensor of Loewner energy, see \citep{Yilin_variation}.
\end{rmrk}
As an immediate corollary, we verify that non-coinciding stress-energy tensor correlation functions defined this way are given by the moments of spectral density, and thus agree with the expressions obtained in \citep{StanfordWittenFermionicLocalization} and \citep*{BLW} by differentiating the partition function. 
This was conjectured and numerically verified in \citep{ConformalBootstrap} for a one-point function (i.e. for $M=1$).
\begin{crl}
Let $M>0$ be an integer, and $\left\{r_p\right\}_{p=1}^M$ be distinct points on the unit circle $\T$. 
Then for any $\sigma>0$,
\begin{multline}
\lim_{\eps_1, \ldots \eps_M\to 0+}\int 6^M \prod_{p=1}^M \left(\obs{\phi}{r_p}{r_p+\eps_p}^2-\eps_p^{-2}-\frac{\sigma^4}{240}\right)\d\FMeas{\sigma^2}(\CnjCl{\phi})\\ 
=
\int_{\R_+} \sigma^{4M} \exp\left(-\frac{\sigma^2 \, k^2}{2}\right)\sinh(2\pi k)\, 2 k^{2M+1} \d k.
\end{multline}
\end{crl}
\begin{rmrk}
We find it interesting that when we turn \eqref{eq_Schw_using_Obs} into the regularisation, proposed in Theorem \ref{thrStressEnergy}, we do not introduce any new diverging constants, as it often happens when one makes sense of differential operators when there is not enough regularity.

It is also worth noting that the introduced shift by $\sigma^4/240$ depends on $\sigma$. However, if we change the normalisation and, instead of fixing the length of the circle and varying the temperature, we fix the temperature and vary the length of the circle (as, e.g., in \citep{ConformalBootstrap}, where the temperature is taken to be $1$ instead of $\sigma^2/2$ here, and the length of the circle is taken to be $\beta$ instead of $1$), this term becomes an absolute constant.
\end{rmrk}

\subsection{Related literature}
The measure corresponding to the Schwarzian Field Theory was
defined in \citep*{BLW} as a finite Borel measure satisfying a natural change of variables formula.
It was further shown that such a measure is unique and can be explicitly constructed following the plan proposed in \citep{BelokurovShavgulidzeExactSolutionSchwarz, BelokurovShavgulidzeCorrelationFunctionsSchwarz}.
Its partition function (i.e. total mass) was defined and computed in the same paper.
For general references about Schwarzian Field Theory and related topics see references therein, as well as \citep{BelokurovShavgulidze3, BelokurovShavgulidze4}.

The cross-ratio correlation functions studied in this work were formally derived in the physics literature in \citep{ConformalBootstrap} by relating Schwarzian Field Theory to a certain degenerate limit of 2D Liouville CFT, and using conformal bootstrap and the DOZZ formula.
Alternatively, the same formulae were formally obtained by connecting JT gravity (which is believed to be holographically dual to the Schwarzian Field Theory) to 2D BF theory and weakly coupled limit of 2D Yang-Mills \citep{Schwarzian_Wilson_Line, JT_Wilson_Line}.
There has been a lot of progress in understanding both 2D Liouville CFT (and conformal bootstrap there) and 2D Yang-Mills from the probabilistic point of view.
Below we briefly review some recent progress in these directions.
However, as of now, there is no rigorous connection between Schwarzian Field Theory and either of these theories from the probabilistic perspective.
It would be interesting to rigorously understand relationships between Schwarzian Field Theories and both 2D Yang-Mills and 2D Liouville CFT as predicted in the physics literature.

Two dimensional Liouville CFT was rigorously constructed as a probability measure on a surface in \citep{LCFT_construction_sphere, LCFT_construction}.
For these probability measures the DOZZ formula was rigorously obtained in \citep{DOZZ}, and the conformal bootstrap solution was derived in \citep{LCFT_conformal_bootstrap, LCFT_Segal}.
See \citep{LCFT_review} and the references therein for the review.
In addition, it is possible to make sense of the stress-energy tensor and derive conformal Ward identities for probabilistic Liouville CFT \citep{LCFT_Stress_Energy1, LCFT_Stress_Energy2, LCFT_Ward}.
The ideas used to make sense of the stress-energy tensor in the context of Liouville CFT are similar to the approaches we take here and in \citep*{BLW}.

Two dimensional Yang-Mills theory, in turn, can also be defined as a probability measure.
One of the main approaches is to view it as a stochastic process of observables given by Wilson loops \citep{Levy_Holonomy, book_Yang_Mills}.
Alternatively, one can also define 2D Yang-Mills as a random distribution \citep{Chevyrev_Yang_Mills}.
Recently, it was proved in \citep{Chevyrev_Yang_Mills_Invar, Chandra_Langevin} that these measures are invariant under the corresponding Langevin dynamics, providing alternative perspective in the context of stochastic quantisation.
In addition, Wilson loops in 2D Yang-Mills are known to satisfy the Makeenko-Migdal equation \citep{Levy_Master_Field, Yang_Mills_MM1, Yang_Mills_MM2}, which, in particular allows one to study certain limits of 2D Yang-Mills measure \citep{Dahlqvist_Yang_Mills_Master_Field}.

Schwarzian Field Theory has also been studied from the algebraic and geometric perspective \citep{AlekseevShatashvili, AlekseevBosonization, AlekseevShatashvili2}. 
From this point of view, it is naturally associated to a particular coadjoint orbit of the Virasoro algebra.
Changing the coadjoint orbit essentially corresponds to changing the regularisation parameter $\alpha$ in \eqref{defMeasureMeasAlpha}, where $\alpha=\pi$ corresponds to the Schwarzian Field Theory.
In \citep{AlekseevBosonization} it was shown that coadjoint orbits corresponding to $\alpha \in i\R$ (which are not $\SL(2, \R)$ invariant) admit global equivariant Darboux charts. 
Physically this amounts to a bosonisation (diagonalisation) of the theory.
As an application, they are able to obtain certain formulae for lower order correlation functions.
We remark that the observables considered in \citep{AlekseevBosonization} are, basically, inverses of $\obs{\phi}{s}{t}_{\alpha}$, as defined in \eqref{defObsAlpha}.

\subsection{Organisation of the paper}
In Section \ref{sectMeasConstructAndProp} we recall the construction of Schwarzian measure, corresponding to the Schwarzian Field Theory, from \citep*{BLW}.

In Section \ref{sect_regularisation} we show how to express correlation functions of $\d\FMeas{\sigma^2}$ as a $\alpha\nearrow \pi$ limit of correlation functions of regularised observables $\obs{\phi}{\cdot}{\cdot}_{\alpha}$ with respect to regularised measures $\d\measN{\sigma^2}{\alpha}$ over $\Diff^1(\T)$, with regularisation parameter $\alpha$ (see Proposition~\ref{prpMainExpectObservRegular}). 
We also show how the derived expressions can be reduced to computing appropriate expectations for the case $\alpha=0$ (see Proposition~\ref{prpObsWeightChange}). 
The latter is computed in Section~\ref{sect_main_calculations}.

All the results from Section~\ref{sect_intro_main_results} are proved in Section~\ref{sect_proofs_main_theorems}.

Appendices \ref{sect_Fourier_Calc}, and \ref{sect_Appendix} contain technical results, which are used throughout the proof.

\subsection{Preliminaries and notations}
\label{sectNotation}
Throughout the paper we will be using the following notations.

\begin{enumerate}

\item The unit circle is denoted by $\T=[0,1]/\{0\sim 1\}$, the non-negative real numbers are denoted by $\R_+ = [0, \infty)$, and for the open disk in the complex plane of radius $r$ we use $\D_r = \left\{z\in \Compl: |z|<r\right\}$.
Moreover, for $s, t\in \T$ we write $t-s$ for the length of the interval going from $s$ to $t$ in the positive direction.
In particular, $t-s\in [0,1)$.

\item We use $\Diff^k(\T)$ for the set of orientation-preserving $C^k$-diffeomorphisms of $\T$, i.e.\ satisfying $\phi'(\tau) > 0$. Note that $\Diff^k(\T)$ is not a linear space.
The topology on $\Diff^k(\T)$ is inherited from the natural topology on $C^k(\T)$.
It turns $\Diff^{k}(\T)$ into a Polish (separable completely metrisable) space as well as a topological group.
The topology on the quotient space $\Diff^{k}(\T)/\SL(2, \R)$ is inherited from $\Diff^{k}(\T)$.

It will also be useful to consider reparametrisations of $[0,1]$, or more general intervals $[0,T]$, whose derivatives are not periodic.
We write $\Diff^k [0, T]$ for the set of orientation-preserving $C^k$-diffeomorphisms of $[0, T]$, i.e.\ satisfying $\phi'(t)>0$, $\phi(0) = 0$, and $\phi(T) = T$.
In particular, the derivatives do not have to match at the endpoints.
The topology on $\Diff^k [0, T]$ is inherited from the natural topology on $C^k[0, T]$.

We further set $\Cfree[0,T] = \left\{f\in C[0,T]\, | \, f(0)=0\right\}$, and $C_0[0,T] = \left\{f\in C[0,T]\, | \, f(0)=f(T) = 0\right\}$, with the topology inherited from $C[0,T]$.

\item We will abuse the notation and for $\phi\in \Diff^1(\T)$ denote its conjugacy class in $\Diff^1(\T)/\SL(2,\R)$ by $\phi$ as well.

\item Throughout the paper we will be encountering expressions of the form $f\big(\arccosh[z]\big)$, for various even analytic functions $f$.
Even though $\arccosh[z]$ is not analytic at $z=0$, the composition $f\big(\arccosh[z]\big)$ still defines an analytic function around $z=0$.
More precisely, we identify $f\big(\arccosh[z]\big)$ with $\widetilde{f}\big(\arccosh^2[z]\big)$, where $\widetilde{f}$ is an analytic function such that $\widetilde{f}(\omega) = f(\sqrt{\omega})$, and  $\arccosh^2[z]$ is the analytic function described in Statement \ref{stmArccoshDef}.

\end{enumerate}

\section{Measure construction and observables}\label{sectMeasConstructAndProp}

The Schwarzian Field Theory was defined in \citep*{BLW} as the unique finite Borel measure which satisfies the expected change of variables formula.
In this section we recall both the rigorous construction of this Schwarzian measure, which is based on the plan from \citep{BelokurovShavgulidzeExactSolutionSchwarz, BelokurovShavgulidzeCorrelationFunctionsSchwarz}, 
and its main properties which are used in the present work.
For proofs and further details and discussions see \citep*{BLW}.

\subsection{Unnormalised Brownian bridge measure}
\label{sec:BB}
The definition of the Schwarzian measure is based on the appropriate reparametrisation of  unnormalised version of the Brownian bridge measure, which we discuss in this subsection.
The unnormalised version of the Brownian bridge measure is a finite measure on $\Cfree[0,T] = \left\{f\in C[0,T]\, | \, f(0)=0\right\}$ formally corresponding to
\begin{equation} \label{e:BB-formaldensity}
\d\WS{\sigma^2}{a}{T}(\xi) = \exp\left\{-\frac{1}{2\sigma^2}\int_{0}^{T}\xi'^{\, 2} (t)\d t\right\} \delta\big(\xi(0)\big) \delta\big(\xi(T)-a\big)\prod_{\tau \in (0,T)}\d\xi(\tau).
\end{equation}

\begin{defn} \label{defn:BB}
  The unnormalised Brownian bridge measure with variance $\sigma^2 > 0$ is a finite Borel measure $\d\WS{\sigma^2}{a}{T}$ on $\Cfree[0, T]$ such that
  \begin{equation}\label{eq:25}
    \sqrt{2\pi T}\sigma \, \exp\left\{\frac{a^2}{2 T\sigma^2}\right\}\d\WS{\sigma^2}{a}{T}(\xi)
  \end{equation}
  is the distribution of a Brownian bridge $\big(\xi(t)\big)_{t\in[0,T]}$ with variance $\sigma^2$ and $\xi(0) = 0$, $\xi(T) = a$.
\end{defn}

For us it will be important that the unnormalised Brownian bridge measure satisfies the following property.
\begin{prp} \label{prop:BBcomposition}
  For any $T_1, T_2>0$, $a\in \R$ and 
any positive continuous functional $F$ on $C[0, T_1+T_2]$, we have
\begin{equation}\label{eqConvWS}
\int F(\xi)\d \WS{\sigma^2}{a}{T_1+T_2}(\xi) = 
\int_{\R}\int \int 
F(\xi_1 \sqcup \xi_2)
\d \WS{\sigma^2}{b}{T_1}(\xi_1)\d \WS{\sigma^2}{a-b}{T_2}(\xi_2)\d b,
\end{equation}
where for $f\in \Cfree [0, T_1]$ and $g\in \Cfree [0, T_2]$, we denote by $f\sqcup g \in \Cfree[0,T_1+T_2]$ 
the function
\begin{equation}
(f\sqcup g)(t) =   
\begin{dcases} 
  f(t) & \text{if }  t\in[0, T_1], \\
   f(T_1)+g(t-T_1)   & \text{if } t\in(T_1, T_1+T_2].
  \end{dcases}
\end{equation}
\end{prp}

\subsection{Measure construction}
\label{sec:MeasConstruction}

In order to define the Schwarzian measure $\d\FMeas{\sigma^2}$,
we first need to construct a finite measure $\mu_{\sigma^2}$ on $\Diff^1(\T)$ which is similar to what is known as the Malliavin--Shavgulidze measure,
see \cite[Section~11.5]{BogachevMalliavin}. 
Formally, this measure corresponds to
\begin{equation} \label{e:MeasureMuFormal}
\d\mu_{\sigma^2}(\phi) 
= \exp\left\{-\frac{1}{2\sigma^2}\int_{0}^{1}\left(\frac{\phi''(\tau)}{\phi'(\tau)}\right)^2\d\tau\right\} \prod_{\tau \in [0,1)}\frac{\d\phi(\tau)}{\phi'(\tau)}.
\end{equation}

We can make sense of this measure by defining it as a push-forward of an unnormalised Brownian bridge on $[0,1]$ with respect to a suitable change of variables.
We \emph{define} $\mu_{\sigma^2}$ by
\begin{equation}\label{defMeasureMu}
\d\mu_{\sigma^2}(\phi) \coloneqq 
\d \WS{\sigma^2}{0}{1}(\xi)
\otimes \d\Theta, \qquad \text{with } \phi(t) = \Theta + \A_{\xi}(t)\enspace (\mathrm{mod}\, 1), \text{ for } \Theta\in [0, 1),
\end{equation}
where $\d\Theta$ is the Lebesgue measure on $[0,1)$ and
\begin{equation} \label{defP}
\A(\xi)(t) \coloneqq \A_{\xi}(t) 
\coloneqq \frac{\int_{0}^t e^{\xi(\tau)}\d\tau}{\int_{0}^1 e^{\xi(\tau)}\d\tau},
\end{equation}
The variable $\Theta$ corresponds to the value of $\phi(0)$.
Note that the map $\xi \mapsto \A(\xi)$ is a bijection between $\Cfree[0, 1]$
and $\Diff^1 [0,1]$ with inverse map
\begin{align} 
\A^{-1}: \Diff^1[0,1] &\to \Cfree[0, 1]\\
\varphi &\mapsto \log \varphi'(\cdot) - \log \varphi'(0).
\end{align}

In view of \eqref{eq:1} and \eqref{e:MeasureMuFormal},
the unquotiented Schwarzian measure is constructed as
\begin{equation} \label{defMeasureMeas}
\d\FMeasSL{\sigma^2}(\phi) = \exp\left\{ \frac{2\pi^2}{\sigma^2}\int_{0}^{1} \phi'^{\, 2}(\tau)\d\tau\right\}\d \mu_{\sigma^2}(\phi).
\end{equation}
Since $\mu_{\sigma^2}$ is supported on $\Diff^1(\T)$, this defines a Borel measure on $\Diff^1(\T)$, 
which turns out to be the unique measure satisfying the natural change of variables formula.

\begin{thr_old}[Theorem~1, \citep*{BLW}]
The measure $\FMeasSL{\sigma^2}$ is the unique (up to a multiplicative constant) $\SL(2,\R)$-invariant Borel measure supported on $\Diff^1(\T)$
  that satisfies the expected change of variables formula
  \begin{equation}  \label{eqDiffeoMeasureChange}
    \frac{\d \psi^{\ast}\!\FMeasSL{\sigma^2}(\phi)}{\d\FMeasSL{\sigma^2}(\phi)} =
    \frac{\d \FMeasSL{\sigma^2}(\psi\circ \phi)}{\d\FMeasSL{\sigma^2}(\phi)} = 
    \exp\left\{
      \frac{1}{\sigma^2}\int_{\T}\Big[ \Schw(\tan(\pi\psi),\phi(\tau))-2\pi^2 \Big]\, \phi'(\tau)^2 \d \tau \right\}
    ,
  \end{equation}
    for  any $\psi\in \Diff^3(\T)$, and has a quotient $\FMeas{\sigma^2} = \FMeasSL{\sigma^{2}}/\SL(2,\R)$ that is a finite Borel measure on $\Diff^1(\T)/\SL(2,\R)$.
\end{thr_old}

\begin{defn}\label{defn_schwarzian}
  The Schwarzian measure is given by $\FMeas{\sigma^2}$.
\end{defn}

One of the main tools which will allow us to compute integrals with respect to $\d\mu_{\sigma^2}$ and $\d\FMeasSL{\sigma^2}$ exactly is the following analogue of Girsanov's Theorem (see \citep*{BLW} for the proof).

\begin{lmm}
\label{crlBBMMeasureChange}
For $g\in \Diff^3[0,1]$ we denote by $L_g$ the left composition operator on $\Diff^1[0,1]$:
\begin{equation}\label{eq:39}
L_g(\phi) = g \circ \phi.
\end{equation}
Suppose $f\in \Diff^3[0,1]$. 
Denote $b = \log f'(1) - \log f'(0)$. 
Let $f^{\sharp} \WS{\sigma^2}{a}{1} = f_{\sharp}^{-1} \WS{\sigma^2}{a}{1}$ be the push-forward of $\WS{\sigma^2}{a}{1}$ under $\A^{-1} \circ L_{f^{-1}} \circ \A=\left(\A^{-1} \circ L_{f} \circ \A\right)^{-1}$.
Then for any $a\in \R$, $f^{\sharp} \WS{\sigma^2}{a}{1}$ is absolutely continuous with respect to $\WS{\sigma^2}{a-b}{1}$ and
\begin{equation}
\frac{\d f^{\sharp} \WS{\sigma^2}{a}{1}(\xi)}{\d  \WS{\sigma^2}{a-b}{1}(\xi)} =  
\frac{1}{\sqrt{f'(0)f'(1)}}
\exp\left\{ 
\frac{1}{\sigma^2}\left[\frac{f''(0)}{f'(0)}\A_{\xi}'(0)-\frac{f''(1)}{f'(1)}\A_{\xi}'(1)\right]
+
\frac{1}{\sigma^2}\int_0^1 \Schw_f \Big(\A_{\xi}(t)\Big)\Big(\A_{\xi}'(t)\Big)^2\d t
\right\}.
\end{equation}
\end{lmm}

\subsection{Measure regularisation}
In order to evaluate expectations with respect to the (finite) quotient measure $\d\FMeas{\sigma^2}$,
we approximate the (infinite) unquotiented measure $\d\FMeasSL{\sigma^2}$  by a particular family of finite measures.
Since these measures are finite, they necessarily break the $\SL(2,\R)$-invariance.
The following regularisation was proposed in \citep{BelokurovShavgulidzeExactSolutionSchwarz} and used in \citep*{BLW} to calculate the partition function.
For $\alpha \in [0,\pi]\cup i\R$ consider the measures given by
\begin{equation}\label{defMeasureMeasAlpha}
\d\measN{\sigma^2}{\alpha} (\phi) = 
\exp\left\{\frac{2 \alpha^2}{\sigma^2} \int_0^{1}\phi'^{\, 2}(t)\d t\right\} 
 \d\WS{\sigma^2}{0}{1}( \xi)
, \qquad \text{where } \phi = \A(\xi).
\end{equation}
In \citep*{BLW} these measures are called \emph{pinned Virasoro $\alpha$-orbital measures},
where pinning refers to the condition that $\phi(0)=0$.
In particular, $\d\FMeasSL{\sigma^2}$ differs from $\d\measN{\sigma^2}{\pi}$ only by rotation by the random angle $\Theta$, which is chosen independently and uniformly on $\T$.
This feature allows us to use these measures to calculate expectations with respect to the Schwarzian measure $\d\FMeas{\sigma^2}$.
\begin{prp}\label{prpIntegralRegularization}
Let $F:\Diff^1(\T)/\SL(2, \R) \to [0, \infty]$ be a continuous function. Then 
\begin{equation}
\int\limits_{\mathclap{\Diff^1(\T)/\SL(2,\R)}} F(\CnjCl{\phi}) \d\FMeas{\sigma^2}(\CnjCl{\phi}) 
=
 \lim_{\alpha \to \pi-} \frac{4\pi(\pi-\alpha)}{\sigma^2}\int\limits_{\mathclap{\Diff^1(\T)}} F(\CnjCl{\phi}) \d\measN{\sigma^2}{\alpha}(\phi),
\end{equation}
where, by slight abuse of notation, we denote the lift of $F$ along the quotient map $\Diff^1(\T) \twoheadrightarrow \Diff^1(\T)/\SL(2,\R)$ by $F$ as well.
\end{prp}
Another advantage of using the constructed measures $\d\measN{\sigma^2}{\alpha}$ is that for them integrals of certain functionals can be computed explicitly, see Section \ref{sect_obs_approximation}.

\section{Observables and their regularisation}\label{sect_regularisation} 

\subsection{Observables}\label{sect:Observables}
The main property of cross-ratio observables \eqref{eqDefObs} is that they define observables on the quotient space $\Diff^1(\T)/\SL(2, \R)$. 

\begin{prp}\label{prpObsSLInvar}
Observables $\obs{\phi}{s}{t}$ are invariant under the action of M\"{o}bius transformations. 
In other words, if $\psi\in \SL(2, \R)$, then
\begin{equation}
\obs{\psi\circ\phi}{s}{t} = \obs{\phi}{s}{t}.
\end{equation}
In particular, they induce well-defined observables on $\Diff^1(\T)/\SL(2, \R)$, which we, slightly abusing the notation, also denote by $ \obs{\phi}{\cdot}{\cdot}$.
\end{prp}
\begin{proof}
Let $\psi\in \SL(2, \R)$ be M\"{o}bius transformation of the unit circle. 
Fix $\phi \in \Diff^1(\T)$ and $s\neq t \in T$. 
Denote $F(\tau) = \tan\big(\pi (\tau-a)\big)$, where $a\in \T$ is such that $a\neq \phi(s)$, and $a\neq \phi(t)$.
Also denote $f = F \circ \phi $ and $g = F \circ \psi \circ F^{-1}$. 
Then
\begin{equation}\label{eqLmmObsSLInvarLineMob}
g (x) = \frac{a x+b}{c x +d},
\end{equation}
for some $a, b, c, d\in \R$ with $ad-bc=1$, and all $x\in \R$.

We have that
\begin{equation}
\obs{\phi}{s}{t} = \left|\obs{f}{s}{t}_0\right|,
\end{equation}
where 
\begin{equation}\label{defObsZero}
\obs{\phi}{s}{t}_{0} 
= 
\dfrac{\sqrt{\phi'(t)\phi'(s)}}{\phi(t)-\phi(s)}.
\end{equation}
It is well-known that the cross-ratio $\obs{f}{\cdot}{\cdot}_0$ is invariant under transformations of the form \eqref{eqLmmObsSLInvarLineMob}. 
Thus,
\begin{equation}
\obs{\phi}{s}{t} 
= \left| \obs{f}{s}{t}_0 \right|
= \left|\obs{g\circ f}{s}{t}_0\right|
=\obs{\psi \circ \phi}{s}{t}.
\end{equation}
\end{proof}

\subsection{Regularised observables}
\label{sect_regularised_obs}
It turns out that in order to be able to compute correlation functions for the measures $\d\measN{\sigma^2}{\alpha}$ we need to modify the observables too.
The new observables will depend on $\alpha$, but they will agree with \eqref{eqDefObs} when $\alpha = \pi$. 
Thus, we will calculate the expectations of new observables against $\d\measN{\sigma^2}{\alpha}$, and take $\alpha\nearrow \pi$.

For $\alpha \in [0, \pi]$ and $0 \leq s \neq t \leq 1$ we define
\begin{equation}\label{defObsAlpha}
\obs{\phi}{s}{t}_{\alpha} = 
\begin{dcases} 
   \dfrac{\alpha\sqrt{\phi'(t)\phi'(s)}}{\sin\big(\alpha[\phi(t)-\phi(s)]\big)} & \text{if } \alpha > 0, \\
   \dfrac{\sqrt{\phi'(t)\phi'(s)}}{\phi(t)-\phi(s)}	      & \text{if } \alpha= 0.
  \end{dcases}
\end{equation}
In this definition we also adopt the convention that $\phi(1)-\phi(0) = 1$, meaning that 
\begin{equation}
\obs{\phi}{0}{1}_{\alpha} = 
\begin{dcases} 
   \dfrac{\alpha \, \phi'(0)}{\sin(\alpha)} & \text{if } \alpha > 0, \\
   \phi'(0)	      & \text{if } \alpha= 0.
  \end{dcases}
\end{equation}

Notice that 
$\obs{\phi}{s}{t}_{\pi} =  \obs{\phi}{s}{t}.$
However, even though it is true that $\lim_{\alpha \to \pi-} \obs{\phi}{t}{s}_{\alpha} = \obs{\phi}{s}{t}_{\pi}$, this limit is not sufficiently uniform in $\phi$, and so it will not allow us to pass to the $\alpha \nearrow \pi$ limit for the integrals of new observables against $\d\measN{\sigma^2}{\alpha}$ (in fact, doing so leads to wrong answer).

Also observe that, generally, for $\alpha<\pi$ we have $\obs{\phi}{s}{t}_{\alpha}\neq \obs{\phi}{t}{s}_{\alpha}$, since, recall, with our notation for $\T$, $\phi(t)-\phi(s)$ and $\phi(s)-\phi(t)$ sum to $1$.
This will be important in Lemma \ref{lmmObsSymmetrizationApprox}, when we exploit this property to approximate $\obs{\phi}{s}{t}_{\pi}$ sufficiently uniformly in $\alpha \nearrow \pi$.

\subsection{Observables approximation}\label{sect_obs_approximation}
Our goal in this subsection is to express correlation functions of $\obs{\phi}{\cdot}{\cdot}$ with respect to $\d\FMeas{\sigma^2}$ as appropriate limits of correlation functions of $\obs{\phi}{\cdot}{\cdot}_{\alpha}$ with respect to $\d\measN{\sigma^2}{\alpha}$, see Proposition \ref{prpMainExpectObservRegular}.
For this, we first need to approximate products of $\obs{\phi}{\cdot}{\cdot}$ with products of $\obs{\phi}{\cdot}{\cdot}_{\alpha}$ uniformly in $\phi$ when $\alpha\to \pi$, see Lemma \ref{lmmObsSymmetrizationApprox}.

\medskip

Let $\left\{s_j\right\}_{j=1}^N$ and $\left\{t_j\right\}_{j=1}^N$ be points on $\T$.
Enumerate the union of points  $\left\{s_j\right\}_{j=1}^N$ and $\left\{t_j\right\}_{j=1}^N$ counter-clockwise as $r_1\leq r_2 \leq \ldots \leq r_{2N}\leq r_{2N+1} = r_1+1$.

For $x \in \T$ define rotation as
\begin{equation}
\Big(\sh{\phi}{x}\Big)(\tau) = \phi(\tau+x) - \phi(x), 
\qquad
\sh{s_j}{x} = \min\left\{s_j-x, t_j-x\right\},
\qquad
\sh{t_j}{x} = \max\left\{s_j-x, t_j-x\right\}.
\end{equation}

\begin{lmm}\label{lmmObsSymmetrizationApprox}
Let $\left\{s_j\right\}_{j=1}^N$ and $\left\{t_j\right\}_{j=1}^N$ be points on $\T$, and $\left\{l_j\right\}_{j=1}^{N}$ be positive numbers. 
Let also $\{r_j\}_{j=1}^N$
be as above.
Then
\begin{multline}
\sum_{m=1}^{2N}\left[\Big(\phi(r_{m+1})-\phi(r_{m})\Big) 
\prod_{j=1}^{N}\obs{\sh{\phi}{r_{m+1}}}{\sh{s_j}{r_{m+1}}}{\sh{t_j}{r_{m+1}}}_{\alpha}^{l_j}\right]\\
\underset{\alpha\to \pi}{=}
\left[\prod_{j=1}^{N}\obs{\phi}{s_j}{t_j}^{l_j}\right]\big(1+o(1)\big),
\end{multline}
uniformly in $\phi \in \Diff^1(\T)$.
\end{lmm}
\begin{proof}
We assume that $\alpha>\pi - \frac{1}{100}$.
Denote
\begin{equation}
s_j^m = r_{m+1}+\sh{s_j}{r_{m+1}}, \qquad t_j^m = r_{m+1}+\sh{t_j}{r_{m+1}} .
\end{equation}

Note that $\forall m, j$,
\begin{equation}
\frac{\obs{\sh{\phi}{r_{m+1}}}{\sh{s_j}{r_{m+1}}}{\sh{t_j}{r_{m+1}}}_{\alpha}}{\obs{\phi}{s_j}{t_j}} 
= \frac{\alpha\,\sin\Big(\pi\big[\phi(t_j^m)-\phi(s_j^m)\big]\Big)}
{\pi\,\sin\Big(\alpha\big[\phi(t_j^m)-\phi(s_j^m)\big]\Big)}.
\end{equation}
Thus, it is sufficient to prove that
\begin{equation}
\sum_{m=1}^{2N}\left[\Big(\phi(r_{m+1})-\phi(r_{m})\Big) 
\prod_{j=1}^{N}\frac{\sin^{l_j}\Big(\pi\big[\phi(t_j^m)-\phi(s_j^m)\big]\Big)}{\sin^{l_j}\Big(\alpha\big[\phi(t_j^m)-\phi(s_j^m)\big]\Big)}\right]
\underset{\alpha\to \pi}{=}
1+o(1),
\end{equation}
We take $\delta = \sqrt{\pi-\alpha} \in (0, 1/10)$.
Let $X = X(\phi) = \left\{m \leq 2N: \Big(\phi(r_{m+1})-\phi(r_{m})\Big) >\delta\right\}$.

Note that since $\sin (y)/y$ is decreasing in $y$ we have
\begin{equation}
\frac{\sin\Big(\pi\big[\phi(t_j^m)-\phi(s_j^m)\big]\Big)}
{\sin\Big(\alpha\big[\phi(t_j^m)-\phi(s_j^m)\big]\Big)}\leq \frac{\pi}{\alpha}.
\end{equation}
Therefore, for any $\phi\in \Diff^1(\T)$,
\begin{equation}
\sum_{m \notin X}\left[
\Big(\phi(r_{m+1})-\phi(r_{m})\Big) 
\prod_{j=1}^{N}\frac{\sin^{l_j}\Big(\pi\big[\phi(t_j^m)-\phi(s_j^m)\big]\Big)}{\sin^{l_j}\Big(\alpha\big[\phi(t_j^m)-\phi(s_j^m)\big]\Big)}\right]
\leq 2 N \left(\frac{\pi}{\alpha}\right)^{\sum_{j=1}^N l_j}\delta .
\end{equation}

Moreover, observe that for $m\in X$ we have $\phi(t_j^m)-\phi(s_j^m)\leq 1-\delta$, so applying Statement \ref{stmSinRatioUniformBound} we get
\begin{align}
\left|\log \prod_{j=1}^{N}
\frac{\sin^{l_j}\Big(\pi\big[\phi(t_j^m)-\phi(s_j^m)\big]\Big)}{\sin^{l_j}\Big(\alpha\big[\phi(t_j^m)-\phi(s_j^m)\big]\Big)}\right|
&\leq \frac{(\pi-\alpha)\sum_{j=1}^N l_j}{\sin(\pi\delta /2)}, \qquad \text{for }m\in X,\\
\sum_{m \in X} \Big(\phi(r_{m+1})-\phi(r_{m})\Big) &\geq 1-2N\delta,
\end{align}
which finishes the proof.
\end{proof}

\begin{prp}\label{prpMainExpectObservRegular}
Let $\left\{s_j\right\}_{j=1}^N$ and $\left\{t_j\right\}_{j=1}^N$ be points on $\T$, and $\left\{l_j\right\}_{j=1}^{N}$ be positive numbers. 
Then
\begin{multline}
\int_{\Diff^1(\T)/\SL(2, \R)}
\left[\prod_{j=1}^{N}\obs{\phi}{s_j}{t_j}^{l_j}\right]\d \FMeas{\sigma^2}\big(\phi\big)\\
=
\lim_{\alpha\to\pi-} \frac{4\pi(\pi-\alpha)}{\sigma^2} 
\int_{\T}\left[  \int_{\Diff^1(\T)} \phi'(0) \prod_{j=1}^{N}\obs{\phi}{\sh{s_j}{x}}{\sh{t_j}{x}}_{\alpha}^{l_j} \d\measN{\sigma^2}{\alpha}(\phi) \right] \d x.
\end{multline}
\end{prp}
\begin{proof}
From Proposition \ref{prpIntegralRegularization} and Lemma \ref{lmmObsSymmetrizationApprox} we get
\begin{multline}\label{eqPrpMainExpectObservRegularFirstStep}
\int_{\Diff^1(\T)/\SL(2, \R)}
\left[\prod_{j=1}^{N}\obs{\phi}{s_j}{t_j}^{l_j}\right]\d \FMeas{\sigma^2}\big(\phi\big)
\\
=
\lim_{\alpha\to\pi-} \frac{4\pi(\pi-\alpha)}{\sigma^2} 
\int_{\Diff^1(\T)} \sum_{m=1}^{2N}\left[\Big(\phi(r_{m+1})-\phi(r_{m})\Big) 
\prod_{j=1}^{N}\obs{\sh{\phi}{r_{m+1}}}{\sh{s_j}{r_{m+1}}}{\sh{t_j}{r_{m+1}}}_{\alpha}^{l_j}\right]
\d \measN{\sigma^2}{\alpha}(\phi)
\end{multline}
Moreover, since
\begin{equation}
\forall x \in (r_m, r_{m+1}]:\qquad
\obs{\sh{\phi}{x}}{\sh{s_j}{x}}{\sh{t_j}{x}}_{\alpha}
=
\obs{\sh{\phi}{r_{m+1}}}{\sh{s_j}{r_{m+1}}}{\sh{t_j}{r_{m+1}}}_{\alpha},
\end{equation}
we get
\begin{multline}
 \Big(\phi(r_{m+1})-\phi(r_{m})\Big) 
\prod_{j=1}^{N}\obs{\sh{\phi}{r_{m+1}}}{\sh{s_j}{r_{m+1}}}{\sh{t_j}{r_{m+1}}}_{\alpha}^{l_j}\\
=
\int_{r_m}^{r_{m+1}}  \phi'(x) \prod_{j=1}^{N}\obs{\sh{\phi}{x}}{\sh{s_j}{x}}{\sh{t_j}{x}}_{\alpha}^{l_j} \d x.
\end{multline}
Combining this with \eqref{eqPrpMainExpectObservRegularFirstStep} and Tonelli's Theorem, we obtain that
\begin{multline}
\int_{\Diff^1(\T)/\SL(2, \R)}
\left[\prod_{j=1}^{N}\obs{\phi}{s_j}{t_j}^{l_j}\right]\d \FMeas{\sigma^2}\big(\phi\big)\\
=
\lim_{\alpha\to\pi-} \frac{4\pi(\pi-\alpha)}{\sigma^2} 
\int_{\T}\left[  \int_{\Diff^1(\T)} \big(\sh{\phi}{x}\big)'(0) \prod_{j=1}^{N}\obs{\sh{\phi}{x}}{\sh{s_j}{x}}{\sh{t_j}{x}}_{\alpha}^{l_j} \d\measN{\sigma^2}{\alpha}(\phi) \right] \d x,
\end{multline}
which together with invariance of $\d\measN{\sigma^2}{\alpha}$ under $\sh{\,}{x}$ finishes the proof. 
\end{proof}

\subsection{Observables and the weight}
Now we reduce the integral with respect to $\d\measN{\sigma^2}{\alpha}$ obtained in Proposition \ref{prpMainExpectObservRegular} to a calculation of an appropriate integral with respect to $\d\measN{\sigma^2}{0}$.
\begin{prp}\label{prpObsWeightChange}
Let $\Big\{\obs{\phi}{s_{j}}{t_j}_{\alpha}\Big\}_{j=1}^{N}$ be a set of observables, and $\left\{l_j\right\}_{j=1}^N$ be real numbers. Then
\begin{multline}
\int_{\Diff^1(\T)} \phi'(0)\prod_{j=1}^N  \obs{\phi}{s_{j}}{t_j}_{\alpha}^{l_j} \d \measN{\sigma^2}{\alpha}(\phi) \\
=
\int_{\Diff^1(\T)}  \obs{\phi}{0}{1}_0 \exp\left\{\frac{8\sin^2\frac{\alpha}{2}}{\sigma^2}\, \obs{\phi}{0}{1}_0 \right\}
 \prod_{j=1}^N  \obs{\phi}{s_{j}}{t_j}_{0}^{l_j} \d \measN{\sigma^2}{0}(\phi) .
\end{multline}
\end{prp}
\begin{proof}
Take 
\begin{equation}
f(t) = \frac{1}{2}\left[
\frac{1}{\tan \frac{\alpha}{2}} \tan\left(\alpha\left(t-\frac{1}{2}\right)\right)+1
\right].
\end{equation}
It is easy to check that
\begin{equation}
\Schw_f(t) = 2\alpha^2,
\qquad
f'(0) = f'(1) = \frac{\alpha}{\sin \alpha},
\qquad
-\frac{f''(0)}{f'(0)} = \frac{f''(1)}{f'(1)} = 2\alpha \tan \frac{\alpha}{2}.
\end{equation}
Thus, it follows from Lemma \ref{crlBBMMeasureChange} that for any non-negative continuous functional $F$ on $\Diff^1(\T)$ we have
\begin{multline} \label{eqLemmaMeasAlphaCalculationMeasureChange}
\int_{\Diff^1(\T)} F(\phi) \d\measN{\sigma^2}{0}(\phi) \\
= \frac{\sin \alpha}{\alpha}  \int_{\Diff^1(\T)} F(f\circ \phi) \,
\exp\left\{-\frac{4\alpha}{\sigma^2} \tan\frac{\alpha}{2} 
\cdot \phi'(0)+\frac{2\alpha^2}{\sigma^2}\int_0^1 \phi'^{\, 2}(t) \d t\right\}
\d\measN{\sigma^2}{0}(\phi).
\end{multline}
Now we choose $F$ to be
\begin{equation}
F(\phi) =  \phi'(0)\,
\exp\left\{\frac{8 \sin^2 \frac{\alpha}{2}}{\sigma^2}\cdot \phi'(0)\right\}
 \prod_{j=1}^N  \obs{\phi}{s_{j}}{t_j}_{0}^{l_j},
\end{equation}
which guarantees that 
\begin{equation}
F(f \circ \phi) = \frac{\alpha}{\sin\alpha}\cdot \phi'(0)\,
\exp \left\{\frac{4\alpha}{\sigma^2} \tan\frac{\alpha}{2} \cdot \phi'(0)\right\}
 \prod_{j=1}^N  \obs{\phi}{s_{j}}{t_j}_{\alpha}^{l_j},
\end{equation} 
whenever $\phi(0)=0$.
Observing that
\begin{equation}
\phi'(0) = \obs{\phi}{0}{1}_0
\end{equation}
 gives the desired claim.
\end{proof}

\section{Main calculations}
\label{sect_main_calculations}
The main goal of this Section is to calculate the integral involving $\obs{\phi}{s}{t}_0$ with respect to $\d\measN{\sigma^2}{0}$ appearing Proposition~\ref{prpObsWeightChange}. 
An important feature of $\obs{\phi}{s}{t}_0$ that we use here is that it is homogeneous in $\phi$. 
In other words, it does not change if multiply $\phi$ by a constant. 
This means that instead of taking $\phi = \A(\xi)$, as in \eqref{defMeasureMeasAlpha}, we can consider $\B(\xi)$, where we define
\begin{equation}
\B(\xi)(t):= \B_{\xi}(t):= \int_{0}^{t} e^{\xi(s)}\d s.
\end{equation}
An advantage of using $\B$ instead of $\A$ is that it is given by a more "local" and simpler expression.
We emphasise again that
\begin{equation}
\obs{\A_{\xi}}{s}{t}_0 = \obs{\B_{\xi}}{s}{t}_0.
\end{equation}
Therefore, in order to calculate the integral from the right-hand side of Proposition \ref{prpObsWeightChange} it is sufficient to study integrals involving $\obs{\B_{\xi}}{\cdot}{\cdot}_0$ with respect to $\d\WS{\sigma^2}{0}{1}(\xi)$.

\subsection{Key Lemma}
The following Lemma is key for the main calculation. 
Informally, it means that when taking integrals with respect to $\d\WS{\sigma^2}{a}{T}$ we can trade exponentials of $\obs{\B_{\xi}}{\cdot}{\cdot}_0$ for the value of parameter $a$.
\begin{lmm}\label{lmmObsExpMomentInsert}
Fix $T>0$.
Let $\left\{s_j\right\}_{j=1}^N$ and $\left\{t_j\right\}_{j=1}^N$ be points on $[0, T]$, and $\left\{l_j\right\}_{j=1}^N$ be positive numbers.
Then for any $z<0$, we have
\begin{equation}
\int 
\exp\left\{
\frac{z}{\sigma^2}\, \obs{\B_{\xi}}{0}{T}_0
\right\}
\prod_{j=1}^N \obs{\B_{\xi}}{s_j}{t_j}^{l_j}_0
\d \WS{\sigma^2}{a}{T}(\xi)
=
\int 
\prod_{j=1}^N \obs{\B_{\xi}}{s_j}{t_j}^{l_j}_0
\d \WS{\sigma^2}{b}{T}(\xi),
\end{equation}
where $b=2\arccosh\left[\cosh\left(\frac{a}{2}\right)-\frac{z}{4}\right]$.
\end{lmm}

\begin{proof}

For $s\in [0, 1]$ define $\widetilde{\xi}(s) = \xi(s\, T)$. 
If $\xi$ is distributed according to $\d \WS{\sigma^2}{p}{T}(\xi)$ for some $p$, then $\widetilde{\xi}$ is distributed according to $\d \WS{\sigma^2 T}{p}{1}(\widetilde{\xi})$.

Notice that
\begin{equation}
\obs{\B_{\xi}}{s}{t}_0 = \frac{1}{T}\obs{\A(\widetilde{\xi})}{s/T}{t/T}_0.
\end{equation}
Therefore,
\begin{multline}
\int 
\exp\left\{
\frac{z}{\sigma^2}\,\obs{\B_{\xi}}{0}{T}_0
\right\}
\prod_{j=1}^N \obs{\B_{\xi}}{s_j}{t_j}^{l_j}_0
\d \WS{\sigma^2}{a}{T}(\xi) \\
=T^{-\sum_{j=1}^N l_j}
\int 
\exp\left\{
\frac{z}{\sigma^2 T}\, \obs{\A(\widetilde{\xi})}{0}{1}_0
\right\}
\prod_{j=1}^N \obs{\A(\widetilde{\xi})}{s_j/T}{t_j/T}^{l_j}_0
\d \WS{\sigma^2 T}{a}{1}(\widetilde{\xi}) 
\end{multline}

Consider $f\in \Diff^3[0,1]$ given by
\begin{equation}
f(t) = \frac{e^{\lambda}\, t}{(e^{\lambda}-1)t+1},
\end{equation}
with $\lambda$ to be chosen later.
It is easy to see that
\begin{equation}
f'(t) = \frac{e^{\lambda}}{((e^{\lambda}-1)t+1)^2},
\qquad
f''(t) = -\frac{2e^{\lambda}(e^{\lambda}-1)}{\big((e^{\lambda}-1) t+1\big )^3},
\qquad
\Schw_f(t) = 0.
\end{equation}
Moreover,
\begin{equation}
\frac{f''(0)}{f'(0)} = 2-2e^{\lambda}, 
\qquad
\frac{f''(1)}{f'(1)} = 2e^{-\lambda}-2.
\end{equation}
Also notice that since $f$ is a fractional linear transformation, we get that observables are invariant under the post-composition with $f$,
\begin{equation}
\obs{\phi}{s}{t}_0 = \obs{f\circ \phi}{s}{t}_0, \qquad \text{for all } \phi\in\Diff^1(\T), s, t\in \T.
\end{equation} 
Therefore, by Lemma \ref{crlBBMMeasureChange} we get
\begin{multline}
\int 
\exp\left\{
\frac{z}{\sigma^2 T}\, \obs{\A(\widetilde{\xi})}{0}{1}_0
\right\}
\prod_{j=1}^N \obs{\A(\widetilde{\xi})}{s_j/T}{t_j/T}^{l_j}_0
\d \WS{\sigma^2 T}{a}{1}(\widetilde{\xi})  \\
= \int 
\exp\left\{
\frac{z}{\sigma^2 T} \, \obs{\A(\widetilde{\xi})}{0}{1}_0
+ \frac{2-2e^{\lambda}}{\sigma^2 T}\, \A'(\widetilde{\xi})(0)  + \frac{2-2e^{-\lambda}}{\sigma^2 T}\,\A'(\widetilde{\xi})(1)
\right\}\\
\times\prod_{j=1}^N \obs{\A(\widetilde{\xi})}{s_j/T}{t_j/T}^{l_j}_0
\d \WS{\sigma^2 T}{a+2\lambda}{1}(\widetilde{\xi}) 
\end{multline}
Notice that if $\widetilde{\xi}\sim \d \WS{\sigma^2 T}{a+2\lambda}{1}(\widetilde{\xi})$, then
\begin{equation}
\A'(\widetilde{\xi})(0) = e^{-\lambda-\frac{a}{2}}\obs{\A(\widetilde{\xi})}{0}{1}_0 ,
\qquad
\A'(\widetilde{\xi})(1) = e^{\lambda+\frac{a}{2}} \obs{\A(\widetilde{\xi})}{0}{1}_0.
\end{equation} 
Therefore,
\begin{equation}
\frac{2-2e^{\lambda}}{\sigma^2 T}\,\A'(\widetilde{\xi})(0)  
+ \frac{2-2e^{-\lambda}}{\sigma^2 T}\,\A'(\widetilde{\xi})(1) 
=
\frac{4\big[\cosh(\lambda+\frac{a}{2})-\cosh(\frac{a}{2})\big]}{\sigma^2 T} \,
\obs{\A(\widetilde{\xi})}{0}{1}_0.
\end{equation}

Taking 
\begin{equation}
\lambda =\arccosh\left[\cosh\left(\frac{a}{2}\right)-\frac{z}{4}\right]-\frac{a}{2}
\end{equation}
we get the desired claim.
\end{proof}

Now we can also prove that observables $\obs{\B_{\xi}}{\cdot}{\cdot}$ have exponential moments.
\begin{prp}\label{prpObsExpMomentBound}
Let $T>0$ and $a\in \R$. For any $0\leq s< t\leq T$ we have
\begin{equation}
\int \exp\left\{\frac{8}{\sigma^2}\, \obs{\B_{\xi}}{s}{t}_0 \right\} 
\d \WS{\sigma^2}{a}{T}(\xi) 
\leq
\frac{1}{\sqrt{2\pi T}\sigma}\exp\left(\frac{2000}{(t-s)\sigma^2}-\frac{a^2}{2 T\sigma^2}\right).
\end{equation}
In particular, the expression above is finite.
\end{prp}

\begin{proof}
We divide the interval $[0, T]$ into three smaller intervals $[0, s],\, [s , t]$ and $[t, T]$. 
Let $\xi_1 \in \Cfree[0, s]$, $\xi_2 \in \Cfree[0, t-s]$, and $\xi_3 \in \Cfree[0, T-t]$.
Notice that 
\begin{align}
\obs{\B(\xi_1\sqcup \xi_2 \sqcup \xi_3)}{s}{t}_0 &= \obs{\B(\xi_2)}{0}{t-s}_0.
\end{align}

Thus, using \eqref{eqConvWS} we get
\begin{multline}
\int 
\exp\left\{\frac{8}{\sigma^2}\, \obs{\B_{\xi}}{s}{t}_0 \right\} 
\d \WS{\sigma^2}{a}{T}(\xi)
= \int_{\R^2}
\left[\int \d \WS{\sigma^2}{u}{s}(\xi_1)\right]\\
\times
\left[\int  \exp\left\{\frac{8}{\sigma^2}\, \obs{\B(\xi_2)}{0}{t-s}_0 \right\} \d \WS{\sigma^2}{v}{t-s}(\xi_2)\right]\\
\times 
\left[\int \d \WS{\sigma^2}{a-u-v}{T-t}(\xi_3)\right]
\d u \d v
\end{multline}
Therefore, it is sufficient to prove the Proposition for $s=0$ and $t=T$.

\medskip

From Lemma \ref{lmmExpMomentAprioriBound} there exists $\eps >0$ such that
\begin{equation}\label{eqAprioriExpMoments}
\int \exp\left\{\eps \, \obs{\B_{\xi}}{0}{T}_0 \right\} 
\d \WS{\sigma^2}{a}{T}(\xi) <\infty.
\end{equation}
Moreover, from Lemma \ref{lmmObsExpMomentInsert} for any $z<0$ we get that 
\begin{equation}\label{eqObsZeroExpMoment}
\int 
\exp\left\{
\frac{z}{\sigma^2}\, \obs{\B_{\xi}}{0}{T}_0
\right\}
\d \WS{\sigma^2}{a}{T}(\xi) 
= \frac{1}{\sqrt{2\pi T}\sigma}\exp\left(- \frac{2}{T \sigma^2} \arccosh^2\left[\cosh\left(\frac{a}{2}\right)-\frac{z}{4}\right]\right).
\end{equation}
Using Lemma \ref{lemmaAnalyticContinuationExponentialMoments} and \eqref{eqAprioriExpMoments} we deduce that the left-hand side is absolutely convergent for all $z\in \D_8$ and the equality above holds for all $z\in \D_8$. Applying Statement \ref{stmArccoshOfCoshLinearBound} we get that for all $z\in[0,8)$,
\begin{equation}
\int 
\exp\left\{
\frac{z}{\sigma^2}\, \obs{\B_{\xi}}{0}{T}_0
\right\}
\d \WS{\sigma^2}{a}{T}(\xi) 
\leq
\frac{1}{\sqrt{2\pi T}\sigma}\exp\left(\frac{2000}{T\sigma^2}-\frac{a^2}{2 T\sigma^2}\right).
\end{equation}
Taking limit $z\to 8$ we finish the proof.
\end{proof}

\subsection{Diagrammatic representation on an interval}

Here we describe a diagrammatic representation for observables $\obs{\phi}{\cdot}{\cdot}_0$ defined on the interval $[0, T]$. This representation is similar to the one described for the circle.

Let $\left\{s_j\right\}_{l=1}^N$  and $\left\{t_j\right\}_{l=1}^N$ be points on $[0, T]$, such that $\forall j: s_j < t_j$.
Let also $\Big\{\obs{\phi}{s_{j}}{t_j}_0\Big\}_{j=1}^{N}$ be a set of observables. 
We are going to represent them as a diagram on the interval $[0, T]$.

We draw the interval $[0, T]$ and all the points $\{s_j\}_{j=1}^N$ and $\{t_j\}_{j=1}^N$ on it. 
For all $1\leq j \leq N$ we connect the point $s_j$ with the point $t_j$ with a half-circle in the upper half-plane.

\begin{defn}
We say that a set of observables $\Big\{\obs{\phi}{s_{j}}{t_j}_0\Big\}_{j=1}^{N}$ is non-interlaced if interiors of all  drawn half-circles on the corresponding diagram are pairwise non-intersecting. In other words, $\forall p, q \in \{1, \ldots, N\}$ we have that one of the following holds: $s_p\leq s_q < t_q\leq t_p$, $s_p< t_p\leq s_q< t_q$, $s_q\leq s_p < t_p\leq t_q$, or $s_q< t_q\leq s_p< t_p$.
\end{defn}

Let $\Big\{\obs{\phi}{s_{j}}{t_j}_0\Big\}_{j=1}^{N}$ be a set of non-interlaced observables. 
The $N$ drawn chords on the corresponding diagram divide the unit disk into $N+1$ connected domains, one of which is unbounded and the $N$ others are bounded.
We number the bounded domains with integers from $1$ to $N$, and assign number $0$ to the unbounded domain.
For each $m \in \{0, \ldots N\}$ we associate a Fourier variable $k_m$ to domain number $m$. We also let $\tau_m$ be the total length of all line line segments (parts of the initial interval $[0, T]$) which form the boundary of $m$-th domain (see Figure \ref{fig:observables_line_diagram_def_fourier}).

For each $j \in \{1, \ldots N\}$ we define $w_1(j)$ and $w_2(j)$ to be the Fourier variables corresponding to the domains that contain the half-circle connecting $s_j$ with $t_j$ in their boundaries (see Figure \ref{fig:observables_line_diagram_obs_fourier}). 
All formulae will be symmetric in $w_1$ and $w_2$, so the exact order does not matter.

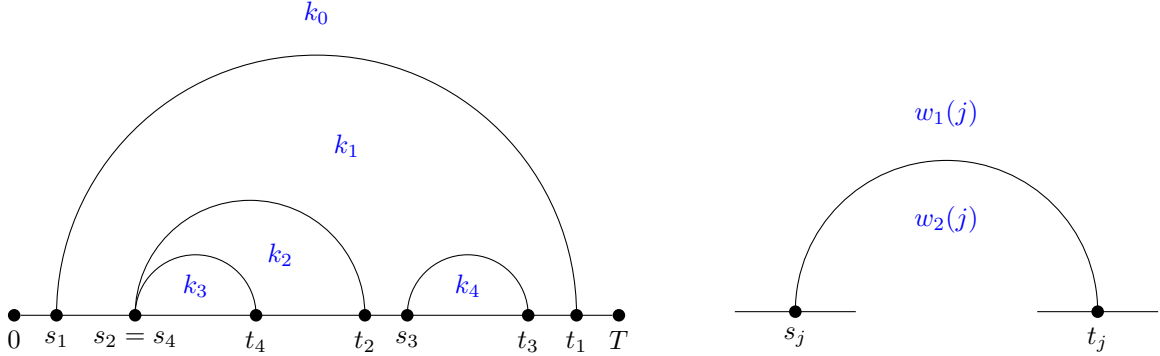
\begin{figure}[tbh]\centering
\begin{subfigure}[t]{0.55\textwidth}
\centering
\begin{tikzpicture}
\newcommand\radius{8}
\newcommand\ra{0.43}
\newcommand\ca{0.5}
\newcommand\rb{0.19}
\newcommand\cb{0.39}
\newcommand\rc{0.1}
\newcommand\cc{0.75}
\newcommand\rd{0.1}
\newcommand\cd{0.3}

	\node[shape=circle, fill=black, scale=0.5,label={-90:$0$}] (n0) at (0, 0) {};
	\node[shape=circle, fill=black, scale=0.5,label={-90:$T$}] (nT) at (\radius, 0) {};
	\node[shape=circle, fill=black, scale=0.5,label={-90:$s_1$}] (ns1) at (\ca *\radius - \ra *\radius, 0) {};
	\node[shape=circle, fill=black, scale=0.5,label={-90:$t_1$}] (nt1) at (\ca *\radius + \ra *\radius, 0) {};
	\node[shape=circle, fill=black, scale=0.5,label={-90:$s_2=s_4$}] (ns2) at (\cb *\radius - \rb *\radius, 0) {};
	\node[shape=circle, fill=black, scale=0.5,label={-90:$t_2$}] (nt2) at (\cb *\radius + \rb *\radius, 0) {};
	\node[shape=circle, fill=black, scale=0.5,label={-90:$s_3$}] (ns3) at (\cc *\radius - \rc *\radius, 0) {};
	\node[shape=circle, fill=black, scale=0.5,label={-90:$t_3$}] (nt3) at (\cc *\radius + \rc *\radius, 0) {};
	\node[shape=circle, fill=black, scale=0.5,label={-90:$t_4$}] (nt4) at (\cd *\radius + \rd *\radius, 0) {};

  	\draw (n0) -- (nT);
  	\draw (\ca *\radius + \ra *\radius, 0) arc  (0:180:\ra *\radius);
  	\draw (\cb *\radius + \rb *\radius, 0) arc  (0:180:\rb *\radius);
  	\draw (\cc *\radius + \rc *\radius, 0) arc  (0:180:\rc *\radius);
  	\draw (\cd *\radius + \rd *\radius, 0) arc  (0:180:\rd *\radius);

  \node[label={90:$\color{\fcolor} k_0$}] (nf0) at ( 0.5* \radius , 0.45* \radius) {};
  \node[label={90:$\color{\fcolor} k_1$}] (nf1) at ( 0.55* \radius , 0.23* \radius) {};
  \node[label={90:$\color{\fcolor} k_2$}] (nf2) at ( 0.44* \radius , 0.05* \radius) {};
  \node[label={90:$\color{\fcolor} k_3$}] (nf3) at ( 0.3* \radius , 0.0* \radius) {};
  \node[label={90:$\color{\fcolor} k_4$}] (nf4) at ( 0.75* \radius , 0.0* \radius) {};

\end{tikzpicture}
\caption{An example of a diagram. Here,\\ 
$\tau_0 = s_1+(T-t_1)$, \\
$\tau_1 = (s_2-s_1)+(s_3-t_2)+(t_1-t_3)$, \\
$\tau_2 = t_2-t_4$, \\
$\tau_3 = t_4-s_4$,\\ 
$\tau_4 = t_3-s_3$.
}\label{fig:observables_line_diagram_def_fourier}
\end{subfigure}
\hspace{0.05\textwidth}
\begin{subfigure}[t]{0.36\textwidth}
\centering
\begin{tikzpicture}
\newcommand\radius{4}
 	\draw (-0.2*\radius, 0) -- (0.2*\radius, 0);
 	\draw (0.8* \radius, 0) -- (1.2*\radius, 0);
  \node[shape=circle,fill=black, scale=0.5,label={-90:$s_j$}] (ns1) at (0, 0) {};
  \node[shape=circle,fill=black, scale=0.5,label={-90:$t_j$}] (nt1) at (\radius, 0) {};
 
	\draw (0,0) arc (180: 0:0.5*\radius);

	\node[label={90:$\color{\fcolor} w_2(j)$}] (nf2) at (0.5*\radius, 0.2*\radius) {};
	\node[label={90:$\color{\fcolor} w_1(j)$}] (nf1) at (0.5*\radius, 0.55*\radius) {};

\end{tikzpicture}
\caption{We use $w_1(j)$ and $w_2(j)$ to denote the Fourier variables $k_m$ corresponding to the domains that lie on both sides of the half-circle that connects $s_j$ with $t_j$. \\
In Figure \ref{fig:observables_line_diagram_def_fourier}, for example, $w_1(2) = k_1$ and $w_2(2) = k_2$ (or vice versa).}
\label{fig:observables_line_diagram_obs_fourier}
\end{subfigure}
\caption{A diagrammatic representation of Fourier variables.}
\end{figure}

\begin{prp}\label{prpLineExpObsFormula}
Fix $T>0$, $a\in \R$, and $N\geq 0$.
Let $\Big\{\obs{\phi}{s_{j}}{t_j}_0\Big\}_{j=1}^{N}$ be a set of non-interlaced observables on $[0, T]$, 
 and $\left\{l_j\right\}_{j=1}^N$ be positive integers. 
Then,
\begin{multline} \label{eqPrpFinalResultLine}
\int 
\prod_{j=1}^N \obs{\B_{\xi}}{s_j}{t_j}^{l_j}_0
\d \WS{\sigma^2}{a}{T}(\xi)
=\int_{\R_+^{N+1}} 
\prod_{j=1}^N \frac{\Gamma\Big(\frac{l_j}{2}\pm i w_1(j) \pm i w_2(j)\Big) }{2 \pi^2\, \Gamma(l_j)} \cdot \left(\frac{\sigma^{2}}{2}\right)^{l_j}\\
\times
\exp\left(-\frac{\tau_0\sigma^2}{2}\cdot k_0^2\right)  \frac{\cos(a\, k_0)}{\pi} \d k_0 \cdot
\prod_{m=1}^{N}\exp\left(-\frac{\tau_m \sigma^2}{2}\cdot k_m^2\right)  \sinh(2\pi k_m)\, 2 k_m \d k_m,
\end{multline}
where the right-hand side converges absolutely.
Moreover, if $\forall j:$ either $s_j\neq 0$ or $t_j\neq T$, then the integral on the right-hand side converges absolutely for any $a\in \D_{2\pi}$.
\end{prp}
\begin{proof}
First, we show the absolute convergence.

From Remark \ref{rmrkGammaFormula} for every $\eps>0$ and any positive integer $l$ there exists $C(\eps,l)$ such that
\begin{equation}
\Gamma\left(\frac{l}{2}\pm ik \pm i\omega \right)\leq C(\eps, l) \exp\Big(-(\pi-\eps)\big(|k|+|w|\big)\Big),\qquad \forall k, \omega\in \R.
\end{equation}
Furthermore, the boundary of the face number $0$ (the unbounded face, corresponding to $k_0$) contains either at least $1$ half-circle or a line segment, and boundary of any other face contains either at least $3$ half-circles or a line segment. 
Therefore, the integral on the right-hand side of \eqref{eqPrpFinalResultLine} converges absolutely.

Moreover, if $\forall j:$ either $s_j\neq 0$ or $t_j\neq T$, then boundary of the face number $0$ (the unbounded face) contains either at least $2$ half-circles or a line segment.  
Thus, the integral on the right-hand side of \eqref{eqPrpFinalResultLine} converges absolutely for all $a\in \D_{2\pi}$.

\medskip

Secondly, we prove \eqref{eqPrpFinalResultLine} by induction on $N$.

\textbf{Base case:} $N = 0$.

In this case $\tau_0 = T$. Thus, the left-hand side of \eqref{eqPrpFinalResultLine} equals
\begin{equation}
\int 
\d \WS{\sigma^2}{a}{T}(\xi) 
= \frac{1}{\sqrt{2\pi T}\sigma} \exp\left(-\frac{a^2}{2T\sigma^2}\right)
\end{equation}
and the right-hand side is equal to
\begin{equation}
\int_{\R_+^1}\exp\left(-\frac{T \sigma^2}{2}\cdot k_0^2\right)
\frac{\cos(a\, k_0)}{\pi}  \d k_0 
= \frac{1}{\sqrt{2\pi T}\sigma}\exp\left(-\frac{a^2}{2 T\sigma^2}\right),
\end{equation}
which proves the base case.

\medskip
\textbf{Induction step:} $N-1\mapsto N$.
By renumbering observables we can assume that $\forall j<N$ we have $s_N\leq s_j$.

There are two possible cases. 

\medskip
\textit{Case 1:} $s_N = 0$ and $t_N = T$.
Renumber Fourier variables so that $w_1(N) =k_0$ and $w_2(N) =k_1$.
Using Lemma \ref{lmmObsExpMomentInsert}, and induction hypothesis we obtain for all $z<0$,
\begin{multline}\label{eqPrpFinalResultLineExponentInduction}
\int 
\exp\left\{
\frac{z}{\sigma^2}\, \obs{\B_{\xi}}{s_N}{t_N}_0
\right\}
\prod_{j=1}^{N-1} \obs{\B_{\xi}}{s_j}{t_j}^{l_j}_0
\d \WS{\sigma^2}{a}{T}(\xi)\\
=
\int_{
\R_+^{N}
} \left[
\prod_{j=1}^{N-1}\frac{\Gamma\Big(\frac{l_j}{2}\pm i w_1(j) \pm i w_2(j)\Big)}{2\pi^2\, \Gamma(l_j)}
\cdot \left(\frac{\sigma^2}{2}\right)^{l_j}
\right]
\exp\left(-\frac{\tau_1 \sigma^2}{2}\cdot k_1^2\right)
\frac{\cos\big(b\, k_1\big)}{\pi} \d k_1 
\\
\times
\prod_{m=2}^{N}\exp\left(-\frac{\tau_m \sigma^2}{2}\cdot k_m^2\right)  \sinh(2\pi k_m)\, 2 k_m \d k_m.
\end{multline}
where $b=2\arccosh\left[\cosh\left(\frac{a}{2}\right)-\frac{z}{4}\right]$.
Notice that it follows from Proposition \ref{prpObsExpMomentBound} that the left-hand side is analytic for $z\in \D_8$. 
Moreover, as we have already proved, the right-hand side is also absolutely convergent whenever $b \in \D_{2\pi}$.

Applying Proposition \ref{prpKeyArccoshExpansion} we obtain
\begin{equation}\label{eqCosExpansionInPrpLineExpObsFormula}
\cos\big(b\, k_1\big) 
= \cos(a\, k_1) + 2 k_1 \sinh(2\pi k_1 )\int_0^{\infty}\sum_{l=1}^{\infty} 
\frac{\Gamma\Big(\frac{l}{2}\pm i k_1 \pm i k_0 \Big)}{2\pi^2\, \Gamma(l)}
\cdot \frac{z^l}{2^l l!}\cos(a\, k_0)\d k_0.
\end{equation}
Now we substitute this expression into \eqref{eqPrpFinalResultLineExponentInduction}. 
We obtain an absolutely convergent expression, since it is dominated by the same expression with $a=0$  which, as we have already proved, is equal to the left-hand side of \eqref{eqPrpFinalResultLineExponentInduction} with $a=0$ which, in turn, is finite because of Proposition \ref{prpObsExpMomentBound}.

Substituting \eqref{eqCosExpansionInPrpLineExpObsFormula} into \eqref{eqPrpFinalResultLineExponentInduction} and taking the coefficient in front of $z^{l_0}$ we get the desired claim.

\medskip

\textit{Case 2:}  
 $t_N<T$.
Now we divide the interval $[0, T]$ into three smaller intervals $[0, s_N],\, [s_N , t_N]$ and $[t_N, T]$. 
Let $\xi_1 \in \Cfree[0, s_N]$, $\xi_2 \in \Cfree[0, t_N-s_N]$, and $\xi_3 \in \Cfree[0, T-t_N]$.
Notice that 
\begin{align}
\obs{\B(\xi_1\sqcup \xi_2 \sqcup \xi_3)}{s}{t}_0 &= \obs{\B(\xi_2)}{s-s_N}{t-s_N}_0  \qquad &\text{if } s,t\in [s_N, t_N];\\
\obs{\B(\xi_1\sqcup \xi_2 \sqcup \xi_3)}{s}{t}_0 &= \obs{\B(\xi_3)}{s-t_N}{t-t_N}_0  &\qquad \text{if } s,t\in [t_N, T].
\end{align}
Let $X = \{j: s_N\leq t_j\leq t_N\}$ and $Y = \{j: t_N \leq s_j \leq T\}$. Observe that $X\cup Y = \{1, \ldots N\}$, since our observables are non-interlaced.
Therefore, we get that
\begin{multline}
\prod_{j=1}^N  \obs{\B\big(\xi_1\sqcup\xi_2\sqcup\xi_3\big)}{s_j}{t_j}^{l_j}_0\\
= \prod_{j\in X}  \obs{\B(\xi_2)}{s_j-s_N}{t_j-s_N}^{l_j}_0\cdot \prod_{j\in Y}  \obs{\B(\xi_3)}{s_j-t_N}{t_j-t_N}^{l_j}_0.
\end{multline}
Thus, using \eqref{eqConvWS} we get
\begin{multline}
\int 
\prod_{j=1}^N 
\obs{\B_{\xi}}{s_j}{t_j}^{l_j}_0
\d \WS{\sigma^2}{a}{T}(\xi)
= \int_{\R^2}
\left[\int \d \WS{\sigma^2}{b}{s_N}(\xi_1)\right]\\
\times
\left[\int \prod_{j\in X}  \obs{\B(\xi_2)}{s_j-s_N}{t_j-s_N}^{l_j}_0\d \WS{\sigma^2}{c}{t_N-s_N}(\xi_2)\right]\\
\times 
\left[\int\prod_{j\in Y}  \obs{\B(\xi_3)}{s_j-t_N}{t_j-t_N}^{l_j}_0  \d \WS{\sigma^2}{a-b-c}{T-t_N}(\xi_3)\right]
\d b \d c
\end{multline}
Applying induction hypothesis to all three factors on the right-hand side (if $Y = \emptyset$, we apply Case 1 to the second factor) and \eqref{eqFourierConv}
finishes the proof of Case 2.

\end{proof}

A similar Proposition holds if we add an exponential of $ \obs{\B_{\xi}}{0}{T}_0$ in the expectation.
\begin{prp}\label{prpLineExpObsWithMomentGenFuncFormula}
Fix $T>0$, $a\in \R$, and $N\geq 0$.
Let $\Big\{\obs{\phi}{s_{j}}{t_j}_0\Big\}_{j=1}^{N}$ be a set of non-interlaced observables on $(0, T)$,
$\left\{l_j\right\}_{j=1}^N$ be positive integers,
and $\alpha\in[0, \pi)$. 
Then,
\begin{multline} \label{eqPrpFinalResultLineWithExponent}
\int 
\exp\left\{
\frac{8 \sin^2\frac{\alpha}{2}}{\sigma^2}\, \obs{\B_{\xi}}{0}{T}_0
\right\}
\prod_{j=1}^N \obs{\B_{\xi}}{s_j}{t_j}^{l_j}_0
\d \WS{\sigma^2}{0}{T}(\xi)\\
=\int_{\R_+^{N+1}} 
\prod_{j=1}^N \frac{\Gamma\Big(\frac{l_j}{2}\pm i w_1(j) \pm i w_2(j)\Big)}{2\pi^2\Gamma(l_j)} 
\cdot \left(\frac{\sigma^2}{2}\right)^{l_j}
\\
\times
\exp\left(-\frac{\tau_0\sigma^2}{2}\cdot k_0^2\right)  \frac{\cosh(2\alpha\, k_0)}{\pi} \d k_0 \cdot
\prod_{m=1}^{N}\exp\left(-\frac{\tau_m \sigma^2}{2}\cdot k_m^2\right)  \sinh(2\pi k_m)\, 2 k_m \d k_m.
\end{multline}
\end{prp}
\begin{proof}
Applying Lemma \ref{lmmObsExpMomentInsert} and Proposition \ref{prpLineExpObsFormula}
we obtain 
\begin{multline}  
\int 
\exp\left\{
\frac{z}{\sigma^2}\, \obs{\B_{\xi}}{0}{T}_0
\right\}
\prod_{j=1}^N \obs{\B_{\xi}}{s_j}{t_j}^{l_j}_0
\d \WS{\sigma^2}{0}{T}(\xi)\\
=
\int_{\R_+^{N+1}} 
\prod_{j=1}^N \frac{\Gamma\Big(\frac{l_j}{2}\pm i w_1(j) \pm i w_2(j)\Big) }{2\pi^2\, \Gamma(l_j)} \cdot \left(\frac{\sigma^{2}}{2}\right)^{l_j}\\
\times
\exp\left(-\frac{\tau_0\sigma^2}{2}\cdot k_0^2\right)  \frac{\cos(b\, k_0)}{\pi} \d k_0 \cdot
\prod_{m=1}^{N}\exp\left(-\frac{\tau_m \sigma^2}{2}\cdot k_m^2\right)  \sinh(2\pi k_m)\, 2 k_m \d k_m.
\end{multline}
with $b = 2\arccosh[1-\frac{z}{4}] $.
For $z\in \D_8$ the left-hand side is analytic, and the right-hand side is analytic for $b\in \D_{2\pi}$.
Thus, using Statement \ref{stmArccoshDef}, we can analytically continue the equality above for all $z\in [0, 8)$. 
Taking $z = 8\sin^2\frac{\alpha}{2}$ we get $b = \pm  2 i\alpha$, which 
gives the desired claim.
\end{proof}

\section{Proofs of the main results}\label{sect_proofs_main_theorems}
\subsection{Proof of Theorem \ref{thrMainCorrelations}}
\label{sect_Thr_Proof}

\begin{proof}[Proof of Theorem \ref{thrMainCorrelations}]
Using Proposition \ref{prpMainExpectObservRegular} and Proposition \ref{prpObsWeightChange} we get,
\begin{multline}
\int_{\Diff^1(\T)/\SL(2,\R)} 
\left[\prod_{j=1}^{N}\obs{\phi}{s_j}{t_j}^{l_j}\right]\d \FMeas{\sigma^2}\big(\phi\big)
=
\lim_{\alpha\to\pi-} \frac{4\pi(\pi-\alpha)}{\sigma^2} \\
\times
\int_{\T}\left[  \int_{\Diff^1(\T)} \obs{\phi}{0}{T}_0
\exp\left\{
\frac{8 \sin^2\frac{\alpha}{2}}{\sigma^2}\, \obs{\phi}{0}{T}_0
\right\}
\prod_{j=1}^{N}\obs{\phi}{\sh{s_j}{x}}{\sh{t_j}{x}}_{0}^{l_j} \d\measN{\sigma^2}{0}(\phi) \right] \d x.
\end{multline}
Using the fact that
\begin{equation}
 \d\measN{\sigma^2}{0}(\phi) = \d\WS{\sigma^2}{0}{T}(\xi), \qquad \text{where } \phi = \A(\xi),
\end{equation}
and
\begin{equation}
\obs{\A(\xi)}{s}{t}_{0} 
= \obs{\B(\xi)}{s}{t}_{0} 
\end{equation}
we obtain
\begin{multline}
 \int_{\Diff^1(\T)} \obs{\phi}{0}{T}_0
\exp\left\{
\frac{8 \sin^2\frac{\alpha}{2}}{\sigma^2}\, \obs{\phi}{0}{T}_0
\right\}
\prod_{j=1}^{N}\obs{\phi}{\sh{s_j}{x}}{\sh{t_j}{x}}_{0}^{l_j} \d\measN{\sigma^2}{0}(\phi) \\
=
 \int \obs{\B_{\xi}}{0}{T}_0
\exp\left\{
\frac{8 \sin^2\frac{\alpha}{2}}{\sigma^2}\, \obs{\B_{\xi}}{0}{T}_0
\right\}
\prod_{j=1}^{N}\obs{\B_{\xi}}{\sh{s_j}{x}}{\sh{t_j}{x}}_{0}^{l_j}  \d\WS{\sigma^2}{0}{T}(\xi).
\end{multline}
Differentiating \eqref{eqPrpFinalResultLineWithExponent} in $\alpha$ gives the desired claim.
\end{proof}

\subsection{Proof of Proposition~\ref{prpExpMoment}}
\begin{proof}[Proof of Proposition~\ref{prpExpMoment}]
Using Corollary~\ref{crlMainObsMoments}, for 
\begin{equation}
C = C(\sigma,s, t) = \sup_{k>0}\left\{\exp\left(-\frac{\big(1-(t-s)\big)\sigma^2}{2}\cdot k^2 \right)\sinh(2\pi k)\, 2k\right\}
\end{equation} 
and any positive integer $l$ we have
\begin{equation}
\int  \obs{\phi}{s}{t}^{l}\d\FMeas{\sigma^2}(\CnjCl{\phi}) 
\leq 
C 
\int_{\R_+^{2}} 
\frac{\Gamma\big(\frac{l}{2} \pm i k_1 \pm i k_2\big)}{2\pi^2\, \Gamma(l)}\cdot\left(\frac{\sigma^2}{2}\right)^l
\exp\left(-\frac{(t-s)\sigma^2}{2}\cdot k_1^2\right)  
\sinh(2\pi k_1)\, 2 k_1 \d k_1 \d k_2.
\end{equation}
Using Tonelli's Theorem and Proposition~\ref{prpKeyArccoshExpansion} for $\beta = 0$ we get that for any $z\in[0,8)$, 
\begin{equation}
\int  \exp\left\{ \frac{z}{\sigma^2}\, \obs{\phi}{s}{t} \right\} \d\FMeas{\sigma^2}(\CnjCl{\phi})
\leq 
C\int_{\R_+} \exp\left(-\frac{(t-s)\sigma^2}{2}\cdot k^2\right) \cos\left(2k \arccosh\left[1-\frac{z}{4}\right]\right)\d k.
\end{equation}
Take $z = 8\sin^2 \frac{\alpha}{2}$, so that $ \cos\left(2k \arccosh\left[1-\frac{z}{4}\right]\right) = \cosh\left(2k \alpha\right)$. 
Therefore,
\begin{multline}
\int  \exp\left\{ \frac{ 8\sin^2\frac{\alpha}{2}}{\sigma^2}\, \obs{\phi}{s}{t} \right\} \d\FMeas{\sigma^2}(\CnjCl{\phi})
\leq 
C\int_{\R_+} \exp\left(-\frac{(t-s)\sigma^2}{2}\cdot k^2\right) \cosh\left(2k \alpha\right) \d k\\
= C\frac{\sqrt{\pi}}{\sqrt{2(t-s)}\sigma}\exp\left(\frac{2\alpha^2}{(t-s)\sigma^2} \right) .
\end{multline}
Taking $\alpha\to \pi$ finishes the proof.
\end{proof}

The fact that the constant $8/\sigma^2$ in the exponent is optimal (as stated in Remark~\ref{rmrkExpOpt}), essentially, follows from the fact that analyticity of the right-hand side of \eqref{eqObsZeroExpMoment} breaks down at $z = 4 (1+\cosh(a/2)) \approx 8$ for $a\approx 0$.
Below we sketch the proof.

\medskip

Fix $\lambda>8$.
Without loss of generality we can assume that $s=0$.
Let $N=N(\lambda)>1$ be a large number to be chosen later.
Take
$
h(k) 
= 
k^2\exp\left(-N^2\sigma^2 k^2/2 \right) .
$
Then,
\begin{equation}
0 \leq h(k) \leq 
\exp\left(-\frac{\big(1-t\big)\sigma^2}{2}\cdot k^2 \right)
\sinh(2\pi k)\, k, \qquad \forall k\in \R.
\end{equation}

Using Corollary~\ref{crlMainObsMoments} and Proposition~\ref{prpLineExpObsFormula} we can deduce that
\begin{equation}\label{eqRmrkOptExpProof}
\int  \exp\left\{ \frac{\lambda}{\sigma^2}\, \obs{\phi}{0}{t} \right\} \d\FMeas{\sigma^2}(\CnjCl{\phi})
\geq
\int_{\R} 
\int \exp\left\{\frac{\lambda}{\sigma^2}\, \obs{\B_{\xi}}{0}{t}_0 \right\} 
\d \WS{\sigma^2}{u}{t}(\xi) \,
\widehat{h}(u)\d u,
\end{equation}
whenever negative part of the integral in right-hand side is finite. 
Here, $\widehat{h}$ is the Fourier transform of $h$, see \eqref{eqFourierDef} for the normalisation.

Notice that $\widehat{h}(u)>0$ for $u\in (-N, N)$.
Using \eqref{eqObsZeroExpMoment} one can show that if $N$ is large enough, then
\begin{equation}
\int \exp\left\{\frac{\lambda}{\sigma^2}\, \obs{\B_{\xi}}{0}{t}_0 \right\} 
\d \WS{\sigma^2}{u}{t}(\xi) 
\end{equation}
is equal to $+\infty$ for $u$ sufficiently small (which follows from the fact that analyticity of the right-hand side of \eqref{eqObsZeroExpMoment} breaks down at $z = 4 (1+\cosh(a/2))$, which is smaller than $\lambda$ for small $a$), and bounded for $|u|>N$.
This means that the right-hand side of \eqref{eqRmrkOptExpProof} is equal to $+\infty$, finishing the proof.

\subsection{Proof of Theorem \ref{thrMainUniq}}
\begin{proof}[Proof of Theorem \ref{thrMainUniq}]
Denote $\DiffG^1(\T)= \left\{\phi\in \Diff^1(\T)\,\big| \, \phi(0)=0, \phi'(0)=1, \phi(1/2) = 1/2\right\}$. 
It is easy to see that as a topological space $\DiffG^1(\T)$ is isomorphic to  $\Diff^1(\T)/\SL(2, \R)$. 
Let $\TechMeas'$ and $\FMeas{\sigma^2}' $ be the push-forward measures of $\TechMeas$ and $\FMeas{\sigma^2}$ to $\DiffG^1(\T)$ respectively.
Then,
\begin{equation}
\int \prod_{j=1}^N \obs{\phi}{s_j}{t_j}^{l_j} \d\FMeas{\sigma^2}'(\phi)
= \int \prod_{j=1}^N \obs{\phi}{s_j}{t_j}^{l_j} \d\TechMeas'(\phi).
\end{equation}

Now we show that for all Borel $A\subset \DiffG^1(\T)$ we have 
$\FMeas{\sigma^2}'(A) =  \TechMeas'(A)$.
It is sufficient to prove this for all cylinder sets.

Take any set of points $0=t_0<t_1<\ldots <t_{M-1}< t_{M}= 1/2$ and $1/2<t_{M+1}<\ldots <t_{2M-1}< t_{2M} = 1$. 
Consider observables $\Big\{\obs{\phi}{0}{t_j}\Big\}_{j=1}^{2M-1}$ and $\Big\{\obs{\phi}{t_{j}}{t_{j+1}}\Big\}_{j=0}^{2M-1}$. 
This set of observables is non-interlaced.

Observe that for all $\phi\in \DiffG^1(\T)$ and $1<j<2M-1$, using $\sin(a-b) = \sin(a)\cos(b)-\sin(b)\cos(a)$, we get
\begin{equation}
\frac{\obs{\phi}{0}{t_j} \obs{\phi}{0}{t_{j+1}}}{\pi \,\obs{\phi}{t_j}{t_{j+1}}} 
= \frac{\phi'(0) \sin\Big(\pi\big[\phi(t_{j+1})-\phi(t_j)\big]\Big)}
{\sin\big(\pi\phi(t_{j+1})\big) \sin\big(\pi\phi(t_{j})\big)}
= \cot \big(\pi\phi(t_{j})\big) -\cot \big(\pi\phi(t_{j+1})\big).
\end{equation}
Moreover, for any $\phi\in \DiffG^1(\T)$ we have $\cot \big(\pi\phi(t_{M})\big) = 0$.
Thus, for all $0\leq j\leq 2M-1$ it is possible to express $\phi(t_j)$ in terms of observables in question using algebraic operations and trigonometric functions. 
Since for all $\phi\in \DiffG^1(\T)$ and $0< j\leq 2M-1$ ,
\begin{equation}
\obs{\phi}{0}{t_j} = \frac{\pi\sqrt{\phi'(t_j)}}{\sin\big(\pi\phi(t_j)\big)},
\end{equation}
we can express $\phi'(t_j)$ in terms of our observables too.

Moreover, notice that taking $N=0$ implies that the total mass of both measures is equal to $\mathcal{Z}(\sigma)$.

\medskip

Now take two random variables $\phi_1 \sim \frac{1}{\mathcal{Z}(\sigma)}\TechMeas'$ and $\phi_2 \sim \frac{1}{\mathcal{Z}(\sigma)}\FMeas{\sigma^2}'$. 
Notice that joint distribution of observables $\Big\{\obs{\phi}{0}{t_j}\Big\}_{j=1}^{2M-1}\cup\Big\{\obs{\phi}{t_{j}}{t_{j+1}}\Big\}_{j=0}^{2M-1}$ is the same for both $\phi_1$ and $\phi_2$, since their correlation functions are equal, and all variables have exponential moments (hence, their characteristic functions are equal).
Moreover, we have already shown that we can express both $\{\phi(t_j)\}_{j=0}^{2M-1}$ and $\{\phi'(t_j)\}_{j=0}^{2M-1}$ as continuous functions of observables in question. 
Therefore, the induced distributions on $\{\phi(t_j),\phi'(t_j) \}_{j=0}^{2M-1}$ are equal. 
This finishes the proof.
\end{proof}

\subsection{Proof of Theorem \ref{thrStressEnergy}}
Theorem \ref{thrStressEnergy} follows immediately from the following Lemma.
 
\begin{lmm}
For $\eps\to 0$ we have 
\begin{multline}
\exp\left(\frac{\eps\,\sigma^2}{2} \cdot k_2^2 \right)  \int_{\R_+} \exp\left(-\frac{\eps\, \sigma^2}{2} \cdot k_1^2 \right) \sinh(2\pi k_1) \, \frac{\Gamma\left(1\pm i k_1 \pm i k_2 \right)}{2\pi^2 \Gamma(2)} \, 2 k_1\d k_1 \\
=
\eps^{-2} + \frac{\sigma^4}{240} + \frac{\sigma^4 k_2^2}{12} + (1+k_2^5)\, O(\eps^{1/2}).
\end{multline}
\end{lmm}
\begin{proof}
Using the fact that $|\Gamma(1+ix)|^2 = \pi x/\sinh(\pi x)$, and  
\begin{equation}
\frac{\sinh(2\pi k_1)}{\sinh\big(\pi (k_1-k_2)\big)\sinh\big(\pi (k_1+k_2)\big)}  
=
\coth\big(\pi(k_1-k_2)\big) + \coth\big(\pi(k_1+k_2)\big),
\end{equation}
we get
\begin{multline}
\int_{\R_+}\exp\left(-\frac{\eps\, \sigma^2}{2} \cdot k_1^2 \right) \sinh(2\pi k_1) \, \frac{\Gamma\left(1\pm i k_1 \pm i k_2 \right)}{2\pi^2 \Gamma(2)} \, 2 k_1\d k_1 \\
=
\int_{\R}\exp\left(-\frac{\eps\, \sigma^2}{2} \cdot k_1^2 \right)k_1 (k_1^2-k_2^2) \coth\big(\pi(k_1-k_2)\big)  \d k_1.
\end{multline}
It is well-known (e.g. see \citep[17.23.21]{GradshteynRyzhik})
\begin{equation}
\mathcal{F}[\coth(\pi x) ](w) = - i\coth\left(\frac{w}{2}\right),
\end{equation}
where by $\mathcal{F}$ we denote the Fourier Transform (in the sense of tempered distribution, e.g. see \citep[Chapter I.3]{SteinWeiss}) normalized as in \eqref{eqFourierDef}.
Therefore, by Plancherel identity 
\begin{multline}
\int_{\R}\exp\left(-\frac{\eps\, \sigma^2}{2} \cdot k_1^2 \right)k_1^{2n+1} \coth\big(\pi(k_1-k_2)\big) \d k_1\\
=
(-1)^n \frac{1}{\sqrt{2\pi}} \left(\eps \, \sigma^2\right)^{-(n+1)}
\int_{\R} \exp\left(-\frac{w^2}{2\eps\, \sigma^2} \right) \He_{2n+1}\left(\frac{w}{\sqrt{\eps}\sigma}\right) \coth\left(\frac{w}{2}\right) e^{i k_2 w}\d w,\\
=
(-1)^n \frac{1}{\sqrt{2\pi}} \left(\eps \, \sigma^2\right)^{-\left(n+\tfrac{1}{2}\right)}
\int_{\R} \exp\left(-\frac{w^2}{2} \right) \He_{2n+1}\left(w\right) \coth\left(\frac{\sqrt{\eps}\sigma \, w}{2}\right) e^{i k_2 \sqrt{\eps}\sigma \, w }\d w,
\end{multline}
where $\He_{k}(x) = (-1)^k e^{x^2/2}\frac{\d^k}{\d x^k } e^{-x^2/2}$ is the $k$-th Hermite polynomial, and the integrals are understood as Cauchy principal value integrals.
It is easy to see that
\begin{align}
 \coth\left(\frac{x}{2}\right) &= \frac{2}{x} + \frac{x}{6}-\frac{x^3}{360}+O(|x|^5),\\
  e^{i k_2 x} &= 1 + ik_2 x - \frac{k_2^2 x^2}{2} - \frac{i k_2^3 x^3}{6} +\frac{k_2^4 x^4}{24} + O(k_2^5 x^5),
\end{align}
where the remainder terms are uniform over $k_2\in\R$.

Therefore,
\begin{multline}
\int_{\R}\exp\left(-\frac{\eps\, \sigma^2}{2} \cdot k_1^2 \right)k_1 \coth\big(\pi(k_1-k_2)\big) \d k_1\\
 = 
\frac{1}{\sqrt{2\pi}} \left(\eps \, \sigma^2\right)^{-1}
\int_{\R} \exp\left(-\frac{w^2}{2} \right)
 \left( 2 + i\, 2\sqrt{\eps}\sigma\, k_2 w + \eps \, \sigma^2\left( \frac{1}{6}-k_2^2\right)w^2 + (1+k_2^3) O(\eps^{3/2} w^3)\right) \d w\\ 
=
2 \left(\eps \, \sigma^2\right)^{-1} + \frac{1}{6}-  k_2^2 + (1+k_2^3)\, O\left(\sqrt{\eps}\right),
\end{multline}
and
\begin{multline}
\int_{\R}\exp\left(-\frac{\eps\, \sigma^2}{2} \cdot k_1^2 \right)k_1^3 \coth\big(\pi(k_1-k_2)\big) \d k_1\\
 = 
 \frac{1}{\sqrt{2\pi}} \left(\eps \, \sigma^2\right)^{-2}
\int_{\R} \exp\left(-\frac{w^2}{2} \right)
 \Bigg( 6-2w^2 + i\ \sqrt{\eps}\sigma (6w-2w^3)\, k_2 + \eps \, \sigma^2 (3w^2-w^4) \left( \frac{1}{6}- k_2^2\right)\\ 
 -i (\eps\,\sigma^2)^{3/2}(3w^3-w^5)\frac{k_2^3}{6}  - (\eps\,\sigma^2)^2(3w^4-w^6)\left(\frac{1}{360}+\frac{k_2^2}{12}-\frac{k_2^4}{12}\right)+ (1+k_2^5)\, O(\eps^{5/2}w^5)\Bigg) \d w\\ 
=
4 \left(\eps \, \sigma^2\right)^{-2}+ \frac{1}{60} + \frac{k_2^2}{2} - \frac{k_2^4}{2}+(1+k_2^5)\, O(\eps^{1/2}).
\end{multline}

Hence,
\begin{multline}
\exp\left(\frac{\eps\,\sigma^2}{2} \cdot k_2^2 \right) \int_{\R_+} \exp\left(-\frac{\eps\, \sigma^2}{2} \cdot k_1^2 \right) \sinh(2\pi k_1) \, \frac{\Gamma\left(1\pm i k_1 \pm i k_2 \right)}{2\pi^2 \Gamma(2)}\left(\frac{\sigma^2}{2}\right)^2 \, 2 k_1\d k_1 \\
=
\eps^{-2}+\frac{\sigma^4}{240}+\frac{\sigma^4 k_2^2}{12} + (1+k_2^5)\, O(\eps^{1/2}),
\end{multline}
where the remainder terms are uniform over $k_2\in\R$.
\end{proof}

\newpage
\appendix

\section{Fourier calculations}
\label{sect_Fourier_Calc}

We use the following normalization for Fourier transform
\begin{align}\label{eqFourierDef}
\widehat{f}(\omega) &= \int_{\R} f(x) e^{-i\omega x}\d x,\\
f(x) &=  \int_{\R} \widehat{f}(\omega)\frac{e^{i\omega x}}{2\pi} \d \omega.
\end{align}
In our normalization
\begin{equation}\label{eqFourierConv}
\int f(y)g(x-y) \d y = \int \widehat{f}(\omega)\widehat{g}(\omega) \frac{e^{i \omega x}}{2\pi} \d \omega.
\end{equation}

The main goal of this Section is to prove the following proposition.
\begin{prp}\label{prpKeyArccoshExpansion}
For any $\beta, k \in \R$ and any $z\in \D_2$ we have
\begin{multline}
\cos\Big(2k\cdot \arccosh\left[\cosh(\beta/2)-z\right]\Big)
=\\
\cos(k\beta)+
2k\sinh(2\pi k)
\int_0^{\infty}
\sum_{l=1}^{\infty} \frac{\Gamma\Big(\frac{l}{2}\pm i k \pm i w\Big)}{2\pi^2 \Gamma(l)}\cdot \frac{(2 z)^l}{l!} 
\cos(w\beta)\d w,
\end{multline}
where the right-hand side converges absolutely.
\end{prp}
\begin{rmrk}
The right-hand side above can also be written as 
\begin{equation}
2k\sinh(2\pi k)
\int_0^{\infty}
\sum_{l=0}^{\infty} \frac{\Gamma\Big(\frac{l}{2}\pm i k \pm i w\Big)}{2\pi^2 \Gamma(l)}\cdot \frac{(2 z)^l}{l!} 
\cos(w\beta)\d w,
\end{equation}
if we interpret $l=0$ term as delta function $\delta(\omega-k)$.
\end{rmrk}
We will deduce it from the following two lemmas.

\begin{lmm}\label{lmmGammaFourier}
For all $k, w\in \R$ and $l>0$ we have
\begin{equation}
\frac{\Gamma(l\pm ik \pm iw)}{\Gamma^2(2l)} = \frac{1}{2}
\int_{\R}\int_{\R} (e^{a/2}+e^{-a/2}+e^{b/2}+e^{-b/2})^{-2l} e^{ika+iwb}\, \d a\, \d b.
\end{equation}
\end{lmm}
\begin{proof}
We start by writing
\begin{equation}
\Gamma(l+ix) = 
\int_0^{\infty} t^{ix+l-1}e^{-t}\d t 
= \int_{-\infty}^{\infty}e^{u(l+ix)} e^{-e^u}\d u.
\end{equation}
Note that the integrand on the right-hand side decays exponentially as $u\to \pm \infty$. 
Applying the equation above for $x \in \{k +w, k -w, -k +w, -k -w\}$ and multiplying them we obtain 
\begin{multline}
\Gamma(l\pm ik \pm iw) 
=
\int_{\R^4} \exp\Big\{-e^{u_1}-e^{u_2}-e^{u_3}-e^{u_4}\Big\}\\
\times\exp\Big\{l(u_1+u_2+u_3+u_4)+ik(u_1+u_2-u_3-u_4)+iw(u_1-u_2+u_3-u_4)\Big\} \d u_1 \d u_2 \d u_3 \d u_4 .
\end{multline}
Changing the variables
\begin{align}
v_1 &= \frac{1}{2} (u_1+u_4), &v_2 = \frac{1}{2} (u_1-u_4), \\
v_3 &= \frac{1}{2} (u_2+u_3), &v_4 = \frac{1}{2} (u_2-u_3), \\
\end{align}
we obtain
\begin{multline}
\Gamma(l\pm ik \pm iw) 
=
4\int_{\R^4} \exp\Big\{-e^{v_1+v_2}-e^{v_1-v_2}-e^{v_3+v_4}-e^{v_3-v_4}\Big\}\\
\times \exp\Big\{2l (v_1+v_3)+2ik (v_2+v_4) + 2iw(v_2-v_4)\Big\}\d v_1 \d v_2 \d v_3 \d v_4.
\end{multline}
Next, we integrate out $v_1$ and $v_3$ using an equality
\begin{equation}
\int_{\R}\exp\big\{-y\,e^{x}\big\}e^{\lambda\, x}\d x =\Gamma(\lambda) y^{-\lambda},
\end{equation}
and get
\begin{multline}
\frac{\Gamma(l\pm ik \pm iw)}{\Gamma^2(2l)}
= 4 \int_{\R^2} \big(e^{v_2}+e^{-v_2}\big)^{-2l}\big(e^{v_4}+e^{-v_4}\big)^{-2l}
\exp\Big\{2ik (v_2+v_4) + 2iw(v_2-v_4)\Big\}
\d v_2\d v_4\\
= 
4 \int_{\R^2} \big(e^{v_2+v_4}+e^{-v_2-v_4}+e^{v_2-v_4}+e^{v_4-v_2}\big)^{-2l}
\exp\Big\{2ik (v_2+v_4) + 2iw(v_2-v_4)\Big\}
\d v_2\d v_4.
\end{multline}
Changing the variables again 
\begin{equation}
a = 2(v_2+v_4),\qquad
b = 2(v_2-v_4),
\end{equation}
we get the desired claim.
\end{proof}

\begin{lmm} \label{lmmSinhRatioFourier}
For all $p, q\in \Compl$ such that $|p|<\pi, |q|<\pi$ we have
\begin{equation}
\frac{\cos(p \omega)-\cos(q \omega)}{\sinh(\pi \omega)} 
= \frac{i}{2\pi}\int_{\R} 
\frac{\sinh(x)\Big(\cosh(p)-\cosh(q)\Big)}{\Big(\cosh(q)+\cosh(x)\Big)\Big(\cosh(p)+\cosh(x)\Big)} 
e^{i x \omega}\d x.
\end{equation}
Notice that the integrand on the right-hand side is exponentially decaying as $x\to \pm\infty$.
\end{lmm}
\begin{proof}
It is well-known (see, e.g. \citep[17.23.20]{GradshteynRyzhik}) that for $a, b\in\R$ with $|a|<|b|$ we have 
\begin{equation}\label{eqGradRyzhSinhFourier}
\frac{\sinh(a\omega)}{\sinh(b\omega)} = \frac{1}{2}\int_{\R} \frac{\sin(\pi a /b)}{b [\cos(\pi a/b)+\cosh(\pi x/b)]} e^{ix\omega }\d x.
\end{equation}
Notice that the formula above also holds true for $a\in \Compl$ with $|a|<|b|$, as both sides are analytic in $a$.

Now we write
\begin{equation}
\frac{\cos(p \omega)-\cos(q \omega)}{\sinh(\pi \omega)} = -\frac{2\sin\left(\frac{(p+q)}{2}\, \omega\right)\sin\left(\frac{(p-q)}{2}\, \omega\right)}{\sinh(\pi \omega)}.
\end{equation}
Thus, combining this with \eqref{eqGradRyzhSinhFourier} for $a = i(p+q)/2$ and $b =\pi$ we obtain
\begin{equation}
\frac{\cos(p \omega)-\cos(q \omega)}{\sinh(\pi \omega)} = \frac{i}{2\pi} 
\int_{\R}\left(\frac{\sinh\left(\frac{p+q}{2}\right)}{\cosh\left(\frac{p+q}{2}\right)+\cosh\left(x-\frac{p-q}{2}\right)}
-\frac{\sinh\left(\frac{p+q}{2}\right)}{\cosh\left(\frac{p+q}{2}\right)+\cosh\left(x+\frac{p-q}{2}\right)}
\right)e^{ix\omega} \d x.
\end{equation}
Moreover,
\begin{multline}
\frac{\sinh\left(\frac{p+q}{2}\right)}
{\cosh\left(\frac{p+q}{2}\right)+\cosh\left(x-\frac{p-q}{2}\right)}
-\frac{\sinh\left(\frac{p+q}{2}\right)}
{\cosh\left(\frac{p+q}{2}\right)+\cosh\left(x+\frac{p-q}{2}\right)} \\
= \frac{\sinh\left(\frac{p+q}{2}\right)\Big(\cosh\left(x+\frac{p-q}{2}\right)-\cosh\left(x-\frac{p-q}{2}\right)\Big)}
{\Big(\cosh\left(\frac{p+q}{2}\right)+\cosh\left(x-\frac{p-q}{2}\right)\Big)
\Big(\cosh\left(\frac{p+q}{2}\right)+\cosh\left(x+\frac{p-q}{2}\right)\Big)}.
\end{multline}
It is easy to show that the numerator on the right-hand side is equal to $\sinh(x) \left(\cosh(p)-\cosh(q)\right),$
and the denominator equals $\left(\cosh(q)+\cosh(x)\right)\left(\cosh(p)+\cosh(x)\right),$
which finishes the proof.
\end{proof}

\begin{proof}[Proof of Proposition \ref{prpKeyArccoshExpansion}]
Using Lemma \ref{lmmSinhRatioFourier} we get
\begin{multline}\label{eqLemmaCosAtccoshFirstFourier}
\frac{\cos\Big(2k\cdot \arccosh\left[\cosh(\beta/2)-z\right]\Big)-\cos(k\beta)}{\sinh(2\pi k)} \\
= \frac{-i}{2\pi}\int_{\R} \frac{z\sinh(x)}{\big(\cosh(x)+\cosh(\beta/2)\big)\big(\cosh(x)+\cosh(\beta/2)-z\big)}e^{2 i x k}\d x.
\end{multline}
Now we Taylor expand in $z$
\begin{equation}
\frac{1}{\cosh(x)+\cosh(\beta/2)-z} 
= \sum_{l=0}^{\infty} \big(\cosh(x)+\cosh(\beta/2)\big)^{-l-1} z^l,
\end{equation}
thus
\begin{equation}
\frac{z\sinh(x)}{\big(\cosh(x)+\cosh(\beta/2)\big)\big(\cosh(x)+\cosh(\beta/2)-z\big)} 
= \sinh(x)\sum_{l=1}^{\infty}\big(\cosh(x)+\cosh(\beta/2)\big)^{-l-1} z^l.
\end{equation}

Differentiating by parts and using Lemma \ref{lmmGammaFourier} we obtain that for any $l\geq 1$,
\begin{multline}\label{eqStmCosArcTermFourier}
\int_{\R}\sinh(x)\big(\cosh(x)+\cosh(\beta/2)\big)^{-l-1}e^{2 i x k} \d x 
= \frac{ik 2^{l+1}}{l} \int_{\R} \Big(e^{x}+e^{-x}+e^{\beta/2}+e^{-\beta/2}\Big)^{-l}e^{2 i x k} \d x \\
= \frac{ik 2^{l}}{\pi l}  \int_{\R} \frac{\Gamma\big(\frac{l}{2}\pm i k\pm i w \big)}{\Gamma^2(l)} e^{iw\beta} \d w.
\end{multline}

In particular, taking $\beta = 0$,
\begin{equation}
\frac{k 2^{l}}{\pi l}  \int_{\R} \frac{\Gamma\big(\frac{l}{2}\pm i k\pm i w\big)}{\Gamma^2(l)} \d w 
\leq \int_{\R}\left|\sinh(x)\right|\big(\cosh(x)+\cosh(\beta/2)\big)^{-l-1} \d x,
\end{equation}
and thus, summing over $l$ we obtain for any $|z|<2$ 
\begin{equation}\label{eqStmCosArcAbsConv}
\frac{k}{\pi} \sum_{l=1}^{\infty} 
\int_{\R} \frac{\Gamma\big(\frac{l}{2}\pm i k\pm i w\big)}{\Gamma(l)}\cdot \frac{(2|z|)^l}{l!} \d w
\leq
\int_{\R} \frac{\left|z\sinh(x)\right|}{\big(\cosh(x)+\cosh(\beta/2)\big)\big(\cosh(x)+\cosh(\beta/2)-|z|\big)}\d x <\infty.
\end{equation}

Now using \eqref{eqStmCosArcTermFourier} again and summing over $l$ we obtain
\begin{multline}
\frac{-i}{2\pi}\int_{\R} \frac{z\sinh(x)}{\big(\cosh(x)+\cosh(\beta/2)\big)\big(\cosh(x)+\cosh(\beta/2)-z\big)}e^{2ix k}\d x \\
= \frac{k}{2\pi^2} \sum_{l=1}^{\infty} 
\int_{\R} \frac{\Gamma\big(\frac{l}{2}\pm i k\pm i w\big)}{\Gamma(l)}\cdot \frac{(2z)^l}{l!} e^{iw\beta} \d w,
\end{multline}
where the expression on the right-hand side converges absolutely because of \eqref{eqStmCosArcAbsConv}.
This, together with \eqref{eqLemmaCosAtccoshFirstFourier} finishes proof.
\end{proof}

\section{Appendix}
\label{sect_Appendix}

Lemmas \ref{lmmExpMomentAprioriBound} and \ref{lemmaAnalyticContinuationExponentialMoments} have already appeared in \citep*{BLW}. 
For the reader's convenience we recall them together with their proofs.

\begin{lmm}[\citep*{BLW}]\label{lmmExpMomentAprioriBound}
Fix $\sigma, a, T>0$.
Then there exists $\eps>0$ such that
\begin{equation}
\int \exp\left(\frac{\eps}{\int_0^T e^{\xi(\tau)}\d\tau} \right)\d\WS{\sigma^2}{a}{T}(\xi)<\infty.
\end{equation}
\end{lmm}
\begin{proof}

Let $\widetilde{\xi}$ be a Brownian bridge distributed according to the probability measure $\sqrt{2\pi T}\sigma\exp\left(\frac{a^2}{2T \sigma^2}\right) \d \WS{\sigma^2}{a}{T}(\widetilde{\xi})$.
Then 
\begin{equation} 
\Prob\left[
\frac{1}{\int_0^{T} e^{\widetilde{\xi}(\tau)}\d \tau} > \lambda
\right] 
\leq 
\Prob\left[
\min_{t\in[0, e \lambda^{-1})]} \widetilde{\xi}(t)<-1
\right]
\leq C^{-1} e^{-C\lambda} ,
\end{equation}
for some $C>0$, independent of $\lambda>10/T$. 
\end{proof}

\begin{lmm}[\citep*{BLW}]\label{lemmaAnalyticContinuationExponentialMoments}
Let $\TechMeas$ be a non-negative measure on $\R_+$. 
Assume that there exists $\eps>0$ such that the exponential moment generating function $F(z) = \int \exp \left(z X\right)\d\TechMeas(X)$ exists for all $z\in [0, \eps)$. 
Assume further that for some $R>0$, $F(z)$ can be analytically continued for all $z\in \D_R$. 
Then $\int \exp \left(z X\right)\d\TechMeas(X)$ converges absolutely for all $z\in \D_R$, and is equal to the analytic continuation of $F(z)$ to $\D_R$.
\end{lmm}
\begin{proof}
Since $\TechMeas$ supported on $\R_+$, by Fubini's Theorem,
\begin{equation}
F(z) = \sum_{n\geq 0} z^n \frac{\int X^n \d\TechMeas(X)}{n!},
\end{equation}
for $z\in [0, \eps)$. Given that $F(z)$ is analytic in $\D_R$, we conclude that the right-hand side converges absolutely for $z\in \D_R$. Thus, by Fubini's Theorem again, $\E \exp \left(z X\right)$ converges for $z\in [0,R)$ and is equal to $F(z)$. We can continue the equality for the whole disk $\D_R$, since $|z^n X^n|\leq |z|^n X^n$, and all expressions are absolutely convergent in $\D_R$.
\end{proof}

\begin{stm}\label{stmSinRatioUniformBound}
For any $\delta\in (0, 1/10)$, 
if $x\in (0, 1-\delta)$, then for any $\alpha \in (\frac{9\pi}{10}, \pi)$ we have
\begin{equation}
\left|\log \left(\frac{\sin(\pi x)}{\sin (\alpha x)}\right) \right|
\leq  \frac{\pi-\alpha}{\sin (\pi\delta/2)}.
\end{equation}
\end{stm}
\begin{proof}
Since $ \sin(y)/y$ is decreasing for $y\in [0, \pi]$, we get that
\begin{equation}
\left|\frac{\d}{\d \alpha}\log\big(\sin (\alpha x)\big) \right|
= \frac{x\left|\cos(\alpha x)\right|}{\sin(\alpha x)}
\leq \frac{\alpha(1-\delta)}{\alpha\sin(\alpha(1-\delta))}
\leq
 \frac{1}{\sin (\pi\delta/2)}.
\end{equation}
\end{proof}

\begin{stm}\label{stmArccoshOfCoshLinearBound}
For any $x\in \R$ we have
\begin{equation}
\arccosh^2\left[
\cosh(x)-2
\right] > x^2-1000.
\end{equation}
\end{stm}
\begin{proof} Since $\arccosh^2\left[
\cosh(x)-2\right]$ is even, it is sufficient to prove the inequality only for $x\geq 0$

For $x\in[0, 10]$ we see that 
\begin{equation}
\arccosh^2\left[
\cosh(x)-2\right]\geq -\pi^2 >x^2-1000.
\end{equation}

For $x\in (10, \infty)$ we have
\begin{equation}
\cosh(x)-2 > \cosh\left(x-\frac{3}{x}\right), 
\end{equation}
since 
\begin{equation}
\frac{\d}{\d y}\cosh(y) = \sinh(y)>y 
\end{equation}
for $y>5$.
Therefore,
\begin{equation}
\arccosh^2\left[
\cosh(x)-2
\right] 
> \left(x-\frac{3}{x}\right)^2 >x^2-10.
\end{equation}
\end{proof}

\begin{stm}\label{stmArccoshDef}
Function $\arccosh^2(z)$ can be analytically continued from $z\in [1, \infty)$ to $z \in \Dom := \left\{z\in\Compl \, \Big| \, \Real z\geq -1 \right\}$. 
Moreover, in this continuation $\arccosh^2(z) = -\arccos^2(z)$ for $z \in (-1, 1)$.
\end{stm}
\begin{proof}
Let $f(z)= \left(z+\frac{1}{z}\right)/2$, and $g(z) = e^z$. 
Notice that $\arccosh^2(z) = \Big[\big(g^{-1} \circ f^{-1}\big)(z)\Big]^2$.

It is easy to see that $f$ is a double cover from $\Compl\backslash (-\infty, 0]$ to $\Compl\backslash (-\infty, -1]$ with only critical point $z=1$, and $g$ is a bijection between $\left\{z\in\Compl \, \Big| \, -\pi < \Im z < \pi \right\}$ and $\Compl\backslash (-\infty, 0]$ with $0$ being mapped to $1$.

Observation that $\big(f \circ g\big)(-z) = \big(f \circ g\big)(z)$ finishes the proof of the first part.

The equality $\arccosh^2(z) = -\arccos^2(z)$ for $z \in (-1, 1)$ follows from the fact that $\omega = i \arccos(z)$ and $\omega = -i \arccos(z)$ are the only solutions of $\cosh(\omega) = z$ in the domain $\left\{\omega\in\Compl \, \Big| \, -\pi < \Im \omega < \pi \right\}$.
\end{proof}

\section*{Acknowledgements}
The author thanks Roland Bauerschmidt, James Norris and Peter Wildemann for many helpful discussions and careful proofreading.

\bibliography{Schwarzian}

\end{document}